\documentclass[a4paper, 11pt]{article}
\usepackage[utf8]{inputenc}

\usepackage[a4paper, total={5.8in, 9.5in}]{geometry}
\usepackage{amsmath,amsfonts, verbatim, amsthm,amssymb,bbm,mathtools,mathrsfs,authblk, esint}

\usepackage[
colorlinks = true,
linkcolor = blue,
anchorcolor = blue,
citecolor = blue]{hyperref}

\newtheorem{theorem}{Theorem}[section]

\newtheorem{lem}[theorem]{Lemma}
\newtheorem{prop}[theorem]{Proposition}
\newtheorem{remark}[theorem]{Remark}
\newtheorem{definition}[theorem]{Definition}

\newtheorem{conjecture}[theorem]{Conjecture}
\numberwithin{equation}{section}
\newcommand{\RR}{\mathbb{R}} 
\newcommand{\fH}{\mathcal{H}} %
\newcommand{\EE}{\mathbb{E}} 
\newcommand{\PP}{\mathbb{P}} 
\newcommand{\fF}{\mathcal{F}} 
\newcommand{\1}{\mathbbm{1}} 
\newcommand{\fL}{\mathcal{L}} 
\newcommand{\fD}{\mathcal{D}} 

\newcommand{\fE}{\mathcal{E}}
\newcommand{\Lp}{\pmb{\Pi}} 
\newcommand{\TT}{\mathbb{T}} 
\newcommand{\MM}{\mathbb{M}} 
\newcommand{\fS}{\mathcal{S}} 
\newcommand{\ZZ}{\mathbb{Z}} 
\newcommand{\HH}{\mathbb{H}} 
\newcommand{\FK}{\Gamma L^2}
\newcommand{\fG}{\mathcal{G}}
\newcommand{\NN}{\mathbb{N}}
\newcommand{\fN}{\mathcal{N}}
\newcommand{\bE}{\mathbf{E}}

\newcommand{\fR}{\mathcal{R}}
\newcommand{\fB}{\mathcal{B}}
\newcommand{\BB}{\mathbb{B}}

\title{Fractional stochastic Landau-Lifshitz Navier-Stokes equations in dimension $d \geq 3$: Existence and (non-)triviality}
\author[1]{Ruhong Jin}
\author[2]{Nicolas Perkowski}
\affil[1]{University of Oxford, \href{mailto:ruhong.jin@maths.ox.ac.uk}{ruhong.jin@maths.ox.ac.uk}}
\affil[2]{Institut für Mathematik, Freie Universität Berlin,
\href{mailto:perkowski@math.fu-berlin.de}{perkowski@math.fu-berlin.de}}

\begin{document}
\maketitle

\begin{abstract}

We investigate fractional stochastic  Navier-Stokes equations in $d\ge 3$, driven by the random force $(-\Delta)^{\frac{\theta}{2}}\xi$ which, as we show, corresponds to a fractional version of the Landau-Lifshitz random force in the physics literature. We obtain the existence and uniqueness of martingale solutions on the torus $\mathbb T^d$ for $\theta > \frac{d}{2}$. For $\theta \le 1$ the equation is supercritical and we regularize the problem by introducing a Galerkin approximation and we study the large scale behavior of the truncated model on $\RR^d$. We show that the nonlinear term in the Galerkin approximation vanishes on large scales when $\theta < 1$ and the model converges to the linearized equation. For $\theta = 1$ the nonlinear term gives a nontrivial contribution to the large scale beahvior, and we conjecture that the large scale behavior is given by a linear model with strictly larger effective diffusivity compared to simply dropping the nonlinear term. The effective diffusivity is explicitly given in terms of the model parameters.
\end{abstract}

\section{Introduction}

The Navier-Stokes equation is a fundamental partial differential equation in fluid dynamics. It is given by 
\[
    \partial_t u = \nu\Delta u - \nabla p - \operatorname{div}(u\otimes u) + f,\qquad \nabla \cdot u = 0,
\]
where $\nu$ is the kinematic viscosity, $f$ is an external force and $p$ is the pressure. 
When the force $f$ is random, we will call it the stochastic Navier-Stokes equation, see for example \cite{flandoli1995martingale,flandoli2008markov,zhu2015three} for noise not rougher than space-time white noise. Even rougher noise, say the derivative of space-time white noise, is considered in the recent works \cite{gubinelli2020hyperviscous, cannizzaro2021stationary} for the vorticity formulation of the two-dimensional fractional stochastic Navier-Stokes equation.

Here we consider a fractional stochastic Navier-Stokes equation in dimension $d \geq 3$, which is given by 
\begin{equation}\label{equ.SNS}
    \partial_t u = -(-\Delta)^{\theta} u - \nabla p -\lambda \operatorname{div}(u\otimes u) + \sqrt{2}(-\Delta)^{\frac{\theta}{2}}\xi,\qquad \nabla \cdot u = 0,
\end{equation}
where $\theta >0$ is a positive number and $\xi$ is a white noise with covariance
\[
    \EE[\xi_i(t,x)\xi_j(s,y)] = \delta_{i,j}\delta(t-s)\delta(x-y),\qquad (t,x),(s,y) \in \RR_{+}\times\TT^d.
\]
For $\theta = 1$, we show that this equation is equivalent to the fluctuating hydrodynamics equation of Landau and Lifshitz~\cite{bandak2022dissipation,landau2013fluid},
\begin{equation}\label{equ.fluctuating_hydrodynamics}
    \partial_t u = \nu\Delta u - \nabla p -\hat{\lambda} \operatorname{div}(u\otimes u) +\nabla \cdot\tau,\qquad \nabla \cdot u = 0,
\end{equation}
with noise term given by $\nabla \cdot\tau$, where $\tau$ is a Gaussian field with covariance
\begin{equation}
    \EE[\tau_{ij}(t,x)\tau_{kl}(t',x')] = \frac{2\nu k_B T}{\rho}(\delta_{ik}\delta_{jl}+\delta_{il}\delta_{jk} - \frac{2}{3}\delta_{ij}\delta_{kl})\delta(x-x')\delta(t-t'),
\end{equation}
where $k_B$ is the Boltzmann constant, $T$ is the temperature, and $\rho$ is the density. 

Equations~\eqref{equ.SNS} and~\eqref{equ.fluctuating_hydrodynamics} are scaling supercritical in $d\ge 3$, and therefore they are outside of the scope of regularity structures~\cite{Hairer2014} or paracontrolled distributions~\cite{Gubinelli2015Paracontrolled}, and making rigorous sense of~\eqref{equ.fluctuating_hydrodynamics} is a long-standing problem, with~\cite{gess2023landau} making recent progress on physical interpretations of truncated versions of the equation.

We do not attempt to make sense of~\eqref{equ.fluctuating_hydrodynamics} directly, and we expect that there is no meaningful solution theory for the non-truncated equation. Instead, we consider two simplified problems. First, we study a hyper-viscous version of the equation, in which the diffusion is replaced by $-(-\Delta)^\theta$ with $\theta>\frac{d}{2}$ and simultaneously the noise is replaced by $\sqrt{2(-\Delta)^\theta} \xi$ to guarantee a fluctuation-dissipation relation that preserves the ``energy measure''. Our first result concerns the generator of equation \eqref{equ.main_SNS} and the corresponding martingale problem. 
\begin{theorem}\label{thm:exi_uni}
    Let $\theta>\frac{d}{2}$ and $M>0$ and consider an initial distribution $\nu$ with $\frac{d\nu}{d\mu^M} \in L^2(\mu)$, where $\mu^M$ is the divergence-free and mean-free space white noise on $\TT^d_M = (\RR^d/(M\ZZ))^d$. There exists a unique-in-law energy solution to~\eqref{equ.SNS} with initial distribution $\nu$.
\end{theorem}
Energy solutions are defined in Section~\ref{sec:energy-solutions} below, where we also give the proof of Theorem~\ref{thm:exi_uni}.

For our second result we focus on the case $\theta \le 1$, so that in particular $\theta < \frac{d}{2}$ and we cannot apply the previous result. Therefore, we consider a truncated equation, with the mollifier $\rho^1$ given by $\fF \rho^1 = \1_{B(0,1)}$ for simplicity
\begin{equation}\label{eq:truncated-SNS-intro}
    \partial_t u = -(-\Delta)^{\theta} u -\lambda \rho^1 * \Lp \operatorname{div}((\rho^1*u)\otimes (\rho^1*u)) + \sqrt{2}(-\Delta)^{\frac{\theta}{2}}\Lp\xi,
\end{equation}
now on $\RR_+ \times \RR^d$; here $\Lp$ is the Leray projection that will be introduced below. We drop the condition $\nabla u = 0$ here but should remember that it holds throughout the paper. This is essentially a Galerkin type projection of the equation for $u$, except that we did not mollify the noise. As will become clear from the proof of Theorem~\ref{thm.solution_truncated} below, we could also mollify $\xi$ with $\rho^1$ without affecting the results, but for simplicity we do not do this. Our second main result concerns the large-scale behavior of $u$: With the scaling
\[
    u^N(t,x) = N^{\frac{d}{2}}u(N^{2\theta}t,Nx)
\]
we obtain (with a new space-time white noise that we still denote by $\xi$ for simplicity)
\begin{equation}\label{eq:scaled-NS}
    \partial_t u^N = -(-\Delta)^{\theta} u^N -\lambda N^{2\theta - \frac{d+2}{2}}\rho^N * \Lp \operatorname{div}((\rho^N*u^N)\otimes (\rho^N*u^N)) + \sqrt{2}(-\Delta)^{\frac{\theta}{2}}\Lp\xi,
\end{equation}
where $\rho^N(x) = N^d\rho^1(Nx)$. 
\begin{theorem}\label{thm:main2}
    Let $\theta \le 1$ and $d \ge 3$ and let $u$ be a stationary solution to~\eqref{eq:truncated-SNS-intro}. Then $u^N$ has weakly convergent subsequences in $C(\RR_+, \fS'(\RR^d))$. For $\theta < 1$, the limit $u^\infty$ is unique and it solves
    \[
        \partial_t u^\infty = -(-\Delta)^{\theta} u^\infty  + \sqrt{2}(-\Delta)^{\frac{\theta}{2}}\Lp\xi,
    \]
    i.e. the nonlinearity simply drops out. For $\theta = 1$, \emph{no} limit solves this equation, i.e. the nonlinearity gives rise to a nontrivial contribution in the limit.
\end{theorem}
This result is shown in Section~\ref{sec:large-scale}, where we conjecture that for $\theta=1$ the limit solves a linear equation, with explicitly given effective diffusivity $>1$. For the mollified Landau-Lifshitz equation~\eqref{equ.fluctuating_hydrodynamics} we would obtain the following linear equation for the limit
\begin{equation*}
    \partial_t u = G(\hat \lambda) \nu \Delta u + \sqrt{G(\hat \lambda)}\sqrt{\frac{2\nu k_B T}{\rho}} \Lp(-\Delta)^{\frac{1}{2}}\xi,
\end{equation*}
where 
\[
    G(\hat{\lambda}) = \sqrt{1 + \frac{\hat{\lambda}^2k_BT\omega_d}{4\nu^2\rho\pi^2(d-2)}},
\]
with an explicit constant $\omega_d$ that depends on the specific mollifier, see Section~\ref{sec:large-scale} for details.

Let us comment briefly on our methods. For Theorem~\ref{thm:exi_uni} we use the energy solution approach introduced and developed in the works \cite{Goncalves2014, gubinelli2013regularization,gubinelli2020infinitesimal,gubinelli2018energy}. This approach is generalized in~\cite{grafner2023energy} to provide a general framework to make sense of the infinitesimal generator of stochastic Burgers and other types of singular stochastic partial differential equations. The crucial feature of the equation which makes this approach works is that it contains a good symmetric part and the singularity is contained in the anti-symmetric part, and this allows us to apply the results of~\cite{grafner2023energy}. In fractional problems like the ones we study here it is expected that there is an ``energy critical exponent'' $\theta$ above which the energy solution approach works, but below which the theory does not give any information. This parameter does not have to be equal to the scaling critical parameter, and in fact here it is $\theta =\frac{d}{2}$, while the scaling critical parameter is $\theta = \frac{d+2}{4}$. For $d \ge 3$, the energy critical exponent is thus bigger than the scaling critical exponent, and it is not well understood what happens in between (or in fact for $\theta \in (1,\frac{d}{2}]$).



For Theorem~\ref{thm:main2} we follow the approach of~\cite{cannizzaro20212d,cannizzaro2021stationary}. In particular, the approach and results are almost the same as the ones in~\cite{cannizzaro2021stationary} for the vorticity formulation of the 2d stochastic Navier-Stokes equation, and we employ their method for the stochastic Navier-Stokes equation in all dimension $d\geq 3$.


While the mathematical methods for both our main results existed previously, we believe that our main contribution is the observation that the stochastic Navier-Stokes equation has an invariant energy measure in all dimensions, and that this allows to study Landau-Lifshitz type dynamics with energy solution methods.


\section{Notation and preliminaries}
For $M > 0$ we will use $\MM_M$ to denote the torus $\TT_M^d=(\RR/(M\ZZ))^d$ of length  $M$. When $M = \infty$, we interpret $\MM_\infty$ as the space $\RR^d$. When $M$ is clear from context, we will drop the index $M$ and simply write $\MM$. We write $\mathcal S(\MM_M)$ for $C^\infty(\MM_M)$ if $M<\infty$, and for the Schwartz functions on $\mathbb R^d$ if $M=\infty$, and $\mathcal S'(\MM_M)$ are the (Schwartz) distributions on $\MM_M$. We will write $\fS(\MM, \RR^d)$ and $\fS'(\MM,\RR^d)$ to denote vector-valued test functions and distributions; if it is clear from context we may also omit $\RR^d$ and simply write $\fS(\MM)$ or $\fS'(\MM)$ even in the vector-valued case.
    
    \subsection{Fourier Analysis}
    For $M > 0$ and $\varphi \in \fS(\MM_M)$ we define the Fourier transform at $k \in \ZZ_M^d=\left(\frac{1}{M}\ZZ\right)^d$ (with $\ZZ_\infty^d = \RR^d$) with the normalization
    \[
        \fF(\varphi)(k) = \hat{\varphi}(k) := \int_{\MM}\varphi(x)e_{-k}(x)dx,
    \]
    where $e_k(x) = e^{2\pi\iota kx}$. The inverse Fourier transformation is
    \[
        \varphi(x) = \frac{1}{M^d}\sum_{k\in\ZZ_M^d}\hat{\varphi}(k)e_k(x),
    \]
    where the sum becomes an integral for $M=\infty$, and Parseval's identity is
    \[
        \langle\varphi,\psi\rangle = \frac{1}{M^{d}}\sum_{k\in\ZZ_M^d}\hat{\varphi}(k)\hat{\psi}(-k).
    \]
    The convolution of two functions $\varphi,\psi$ on the space $\MM_M$ is defined by 
    \[
        \varphi *_M \psi(x) := \int_{\MM_M} \varphi(y)\psi(x-y)dy,
    \]
    and the convolution of two functions on $\ZZ_M^d$ is defined by 
    \[
        \hat{\varphi}*_M\hat{\psi}(k) := \frac{1}{M^d}\sum_{k\in\ZZ_M^d} \hat{\varphi}(l)\hat{\psi}(k-l),
    \]
    so that 
    \[
        \widehat{\varphi*_M\psi}(k) = \hat{\varphi}(k)\hat{\psi}(k),\qquad \widehat{\varphi\psi}(k) = \hat{\varphi}*_M\hat{\psi}(k).
    \]
    Derivatives and fractional Laplacian can be expressed as Fourier multipliers:
    \[
        \widehat{\partial_i\varphi}(k) = 2\pi\iota k_i\hat{\varphi}(k),\qquad  \widehat{(-\Delta)^\theta\varphi}(k) =(2\pi|k|)^{2\theta}\hat{\varphi}(k),
    \]
    and the Leray projection $\Lp$ maps vectors $\varphi \in \fS(\MM,\RR^d)$ to divergence free vectors,
    \[
        \widehat{\Lp(\varphi)}_i(k) = \sum_{j}\widehat{\Lp}_{ij}(k)\hat{\varphi_j}(k),
    \]
    where
    \begin{equation}\label{equ.Leray_multiplier}
        \widehat{\Lp}_{ij} = \delta_{ij} - \frac{k_i k_j}{|k|^2}, \qquad 1 \leq i,j\leq n,
    \end{equation}
    with the convention $\widehat{\Lp}_{ij}(0)=\delta_{ij}$. 
    
    \subsection{Gaussian Analysis}
    We consider a mean-free and divergence-free white noise: For $M<\infty$, $\mu^M$ is an $\fS'(\MM,\RR^d)$-valued random variable such that $\mu^M(f)$ is centered Gaussian for all $f\in \fS(\MM,\RR^d)$, with covariance
    \[
        \EE[\mu^M(f_1)\mu^M(f_2)] = \left\langle \Lp\left(f_1 - \fint f_1 dx\right),\Lp\left(f_2 - \fint f_2dx\right)\right\rangle_{L^2(\MM,\RR^d)},
    \]
    for all $f_1,f_2 \in \fS(\MM,\RR^d)$, where $\fint_{\MM_M}\dots dx = \frac{1}{M^d} \int_{\MM_M} \dots dx$. It can be seen that the measure induced on $\fS'$, which we again noted as $\mu^M$, is concentrated on the divergence-free and mean-free fields. We can also consider the Hilbert space
    \[
        \HH_M :=\left\{\varphi \in L^2(\{1,\dots,d\}\times \MM_M): \operatorname{div} (\varphi) = 0, \fint_{\MM}\varphi dx = 0\right\},
    \]
     equipped with the norm $\|h\|^2_{\HH_M} = \sum_{i=1}^d \int_{\MM_M} |h(x)|^2 dx$, so that $\mu^M$ is a white noise on $\HH_M$ with 
    \[
        \EE[\mu^M(h_1)\mu^M(h_2)] = \langle h_1,h_2\rangle_{\HH_M},
    \]
    for $h_1,h_2 \in \HH_M$. For $M=\infty$ we also write $\mu := \mu^\infty$ for the divergence-free white noise on $\RR^d$, i.e. the covariance is
    \[
        \EE[\mu^M(f_1)\mu^M(f_2)] = \left\langle \Lp f_1,\Lp f_2\right\rangle_{L^2(\MM,\RR^d)}.
    \]
    We write $L^2(\fS'(\MM),\mu^M)$ (in short $L^2(\mu^M)$) for the space of all square-integrable random variables with respect to measure $\mu^M$. The $n-$th Wiener Chaos is defined via 
    \[
        \mathscr{H}_{M,n} := \overline{\text{span}\{H_n(\mu^M(h)): h \in \HH_M, \|h\| = 1\}},
    \]
    where $H_n$ are the Hermite polynomials such that $H_0=1$, $H'_n = nH_{n-1}$ and $\EE[H_n(X)] = 0$, $n\ge 1$, for a standard normal random variable $X$. When no confusion arises, we drop the index $M$ for simplicity. 
    
    For all non-negative integers $n$, let $\FK_n$ be the symmetric subspace of $\HH^{\otimes n}$. The Fock space $\Gamma L^2 = \bigoplus_{n=0}^\infty \Gamma L_n^2$ consists of symmetric functions with finite norm defined by 
    \[
        \|\varphi\|_{\FK}^2 := \sum_{n=0}^\infty n!\|\varphi_n\|_{\HH^{\otimes n}}^2,
    \]
    for $\varphi = (\varphi_n)_{n\geq 0}$ with $\|h^{\otimes n}\|_{\HH^{\otimes n}} := \|h\|_{\HH}^n$. The space $\FK$ can be described as 
    \[
        \FK := \{(\varphi_n)_{n\geq 0}: \varphi_n\in\HH^{\otimes n}, \varphi_n \text{ symmetric and }\|\varphi\|_{\FK} < \infty\}.
    \]
    For convenience, we also may say that a non-symmetric function is in the Fock space, provided its symmetrization is in the Fock space. For any $f \in L^2(\mu)$, we have the chaos decomposition 
    \[
        f = \sum_{n=0}^\infty I_n(\varphi_n),\qquad \varphi = (\varphi_n)_n \in \FK,
    \]
    where $I_n$ is the continuous linear map map from the symmetric subspace of $\HH^{\otimes n}$ to $L^2(\mu)$ with $I_n(h^{\otimes n}):=H_n(\mu(h))$ for $h \in \HH$ and $\|h\|_{\HH} = 1$. Furthermore, the identity 
    \[
        \|f\|_{L^2(\mu)} = \|\varphi\|_{\FK}
    \]
    holds. Another tool we need is the identity 
    \begin{equation}
        I_p(\varphi_p)I_q(\varphi_q) = \sum_{r=0}^{p\wedge q}r!\binom{p}{r}\binom{q}{r} I_{p+q-2r}(\widetilde{\varphi_p\otimes_{r}\varphi_q}),
    \end{equation}
    where $\tilde{\cdot}$ is the symmetrization operator and 
    \begin{align*}
        &\varphi_p\otimes_r \varphi_q((i_{1},x_{1}),\dots,(i_{p+q-2r},x_{p+q-2r})) \\
        &\hspace{5pt}= \langle \varphi_p((i_1,x_1),\dots,(i_{p-r},x_{p-r}),\cdot), \varphi_q((i_{p-r+1},x_{p-r+1}),\dots,(i_{p+q-2r},x_{p+q-2r}),\cdot)\rangle_{\HH^{\otimes r}}. 
    \end{align*}
    We will also need ``Sobolev-like'' spaces $\fH^{\alpha}_\beta$ defined by 
    \[
        \fH^{\alpha}_\beta := (1+\fN)^{-\beta}(1-\fL_\theta)^{-\alpha}\FK,
    \]
    where the number operator $\fN$ maps $\varphi = (\varphi_n)_n$ to $\fN \varphi = (n \varphi_n)_n$, and where $\fL_\theta$ is defined in equation \eqref{equ.geneator_Ltheta_cylindrical} below. For $\theta = 1$ and $\varphi \in \FK_1$, this  is exactly the inhomogeneous $H^\alpha$ Sobolev norm.
    
\section{The infinitesimal generator of the truncated equation}
Our goal is to solve the following stochastic Navier-Stokes Equation on the space $\MM$:
    \begin{equation}\label{equ.stochastic_NS}
        \left\{
        \begin{array}{ll}
             \partial_t u = -(-\Delta)^\theta u - \lambda\Lp(\operatorname{div}(u\otimes u)) + \sqrt{2}(-\Delta)^{\frac{\theta}{2}}\Lp\xi,&  \text{ on }\RR^+\times \MM,\\
             \operatorname{div}(u) = 0, & \text{ on }\RR^+\times \MM.
        \end{array}
        \right.
    \end{equation}
    Here $u:\MM\rightarrow \RR^d$ is a divergence-free vector function, $\xi = (\xi_1,\dots,\xi_d)$ are independent space-time white noises on $\{1,\dots,d\}\times \MM$, and we recall that $\Lp$ is the Leray projector.

    We start by considering the following approximation: Let $\chi$ be a smooth, even function  on the whole space $\RR^d$ with support on the ball $B(0,1)$ and consider the kernel $\rho^{MN}$ with Fourier transform $\hat \rho^{MN}(k) = \chi(k)$ for $k \in \ZZ^d_{MN}$ (where $MN$ is the product $M\cdot N$). Let $u^{MN}$  be the divergence-free solution to the equation
    \begin{equation}
        \partial_t u^{MN} = -A^\theta u^{MN} - \lambda\rho^{MN}*_{MN}\Lp(\operatorname{div}((\rho^{MN}*_{MN} u)\otimes(\rho^{MN}*_{MN} u))) + \sqrt{2}A^{\frac{\theta}{2}}\Lp\xi,
    \end{equation}
    on $\RR^+\times \MM_{MN}$, for divergence-free initial condition $u^M_0$. With the same scaling as in~\eqref{eq:scaled-NS} we see that $u^{N,M}:=N^{\frac{d}{2}}u^{MN}(N^{2\theta}t,Nx)$ solves
    \begin{equation}\label{equ.Stochastic_NS_torus_truncated}
        \partial_tu^{N,M} = -A^\theta u^{N,M} - \lambda_{N,\theta}B^{N,M}(u^{N,M}) + \sqrt{2}A^{\frac{\theta}{2}}\Lp\xi,
    \end{equation}
    on $\RR^+\times \MM_M$ with divergence-free initial condition $u^{N,M}(0) = u_0^{N,M}$, where $A = -\Delta$, $\lambda_{N,\theta} = \lambda N^{2\theta - \frac{d+2}{2}}$ and 
    \begin{equation}
        B^{N,M}(u) = \rho^{N,M}*_M\Lp(\operatorname{div}((\rho^{N,M}*_Mu)\otimes(\rho^{N,M}*_Mu))),
    \end{equation}
    where $\widehat{\rho^{N,M}}(k) = \chi(\frac{k}{N})$ on $\ZZ_M^d$. Next, let us define solutions to  equation \eqref{equ.Stochastic_NS_torus_truncated}. 
    \begin{definition}\label{def.SNS_weak}
        A divergence-free stochastic process $u^{N,M} \in C(\RR,\fS'(\MM,\RR^d))$ is called a strong solution to $\eqref{equ.Stochastic_NS_torus_truncated}$  with $\lambda_{N,\theta}$ replaced by $\hat{\lambda} > 0$, if for all test function $\varphi \in \fS(\MM_M, \RR^d)$, we have 
        \begin{equation}\label{equ.SNS_trunc_strong_sol}
            u^{N,M}_t(\varphi) = u_0^{N,M}(\varphi) - \int_0^t u^{N,M}_s(A^\theta\varphi)ds - \hat{\lambda}\int_0^tB^{N,M}(u^{N,M}_s)(\varphi)ds + M_t^{\varphi},
        \end{equation}
        where $M_t^\varphi = \sqrt{2}\xi(\1_{[0,t]}\otimes A^{\frac\theta2} \Lp \varphi)$ is a continuous martingale with quadratic variation 
        \[
            [M^\varphi]_t = 2t\|A^{\frac{\theta}{2}}\Lp\varphi\|_{\HH_M}^2.
        \]
    \end{definition}
    \begin{remark}
        We may consider Schwartz functions as test functions. However, since the solution $u^{N,M}$ is divergence-free, we actually have
        \[
            u^{N,M}_t(\varphi) = (\Lp u^{N,M}_t)(\varphi) = u^{N,M}_t(\Lp\varphi).
        \]
        Thus we can restrict to divergence-free test functions. Furthermore, when $M<\infty$, the $0-$th Fourier coefficient does not change in time. Thus, for convenience, we will also restrict the test function to have average $0$ when $M$ is finite. We will write $\fS_f(\MM)$ and $\fS_f'(\MM)$ to indicate that we choose the divergence-free functions.
    \end{remark}

    To start with, let us give an orthonormal basis of $\HH_M$.
    \begin{prop}\label{lem.Leray_proj}
        Let $d$ be the dimension and let $M<\infty$. Then for each Fourier mode $k \in \ZZ^d_M\setminus\{0\}$, there are exactly $d-1$ vectors $(\tilde{e}_k^i)_{1 \leq i\leq d-1}$ such that $\{\tilde{e}_k^i\}_{k\neq 0, i=0,\dots, d-1}$ forms an orthogonal basis of $\HH_M$.
    \end{prop}
    \begin{proof}
        We start with the $L^2$ Fourier basis: 
        \[
            e_k^i(z):=\{(e_k(z)\delta_{ij})_{j = 1,\dots,d}\}_{0\neq k\in\ZZ_M^d, i = 1,\dots,d},
        \]
        with Fourier transformation
        \[
            \delta_k^i(\bar{k}):=\{(\delta_k(\bar{k})\delta_{ij})_{j = 1,\dots,d}\}_{0\neq k\in\ZZ_M^d, i = 1,\dots,d}.
        \]
        The Fourier transform of the Leray projection of $e_k^i$ is given by 
        \[
            \left(\widehat{\Lp}\delta^i_k\right)(\bar{k}):= (\widehat{\Lp}_{ij}\delta^i_k)_{j=1,\dots,d}(\bar{k})=\left(\delta_{ij} - \frac{\bar{k}_i\bar{k}_j}{|k|^2}\right)_{j=1,\dots,d}\delta_k(\bar{k})
        \]
        This means that the Leray projections for different $k\in\ZZ_M^d$ are linearly independent. We now go to check the case when $k$ is the same but $i$ is different. We only need to compute the rank of the matrix $\widehat{\Lp}(k)$ for any fixed non-zero $k$, where
        \[
            \widehat{\Lp}(k) = \left(
            \begin{array}{cccc}
                |k|^2 - k_1^2 & -k_1k_2 & \dots & -k_1k_d \\
                - k_1k_2 & |k|^2 -k_2^2 & \dots & -k_2k_d \\
                \vdots&\vdots&\ddots&\vdots \\
                -k_dk_1 & -k_dk_2 & \dots & |k|^2-k_d^2
            \end{array}
            \right) = |k|^2 I_d - (k_ik_j)_{i,j}.
        \]
        To show the rank is $d-1$, we only need to show that $(k_ik_j)_{i,j}$ has a one-dimensional eigenspace with respect to the eigenvalue $|k|^2$. One eigenvector for $|k|^2$ is $k$. If $(x_j)$ is another eigenvector for $|k|^2$, then we can take the inner product with $x$ in the equality $k_i \sum_j k_j x_j = |k|^2 x_i$, $i=1,\dots,d$, and we obtain
        \[
            |k|^2|x|^2 = \left(\sum_j x_jk_j\right)^2.
        \]
        By the Cauchy-Schwarz inequality, we know that $x$ is a multiple of $k$ and therefore the dimension of the eigenspace is exactly one. Thus we proved the proposition.
    \end{proof}
    Now we can prove the existence and uniqueness of strong solutions to the truncated equation:

    \begin{theorem}\label{thm.solution_truncated}
        For $M< \infty$, equation $\eqref{equ.Stochastic_NS_torus_truncated}$ admits a unique strong solution $u^{N,M} \in C(\RR_+,\BB^{-\frac{d}{2}-}_{2,2})$ for any divergence-free initial condition $u_0\in\BB^{-\frac{d}{2}-}_{2,2}$, where 
        \[
            \BB^{-\frac{d}{2}-}_{2,2} = \bigcap_{\alpha < -\frac{d}{2}} \BB^{\alpha}_{2,2}.
        \]
        Furthermore, the solution is a strong Markov process and $\mu^M$ is an invariant measure.
    \end{theorem}
    
    \begin{proof}
        In this proof we drop the superscript $\cdot^M$. We can follow Section 4 of \cite{gubinelli2013regularization} or Proposition 2.1 of \cite{cannizzaro2021stationary}. As indicated by Proposition \ref{lem.Leray_proj}, instead of considering Fourier series, we take the orthogonal basis of $\HH$ as test functions, i.e. $\tilde{e}_k^i$ in Proposition \ref{lem.Leray_proj}. We can then decompose $u^{N} = v^{N} + Z^{N}$ where $v^{N}$ and $Z^{N}$ are independent with disjoint spectral support. Indeed, for the decomposition we take $V^N:= \{0\neq k\in\ZZ_M^d: \frac{k}{N}\in \operatorname{supp}(\chi)\}$ and we let $v^N$ be the orthogonal projection of $u^{N}$ to the space generated by $\{\tilde{e}_k^i\}_{k\in V^N, i=0,\dots, d-1}$.
        
        Furthermore, $Z^{N}$ is the Leray projection of the infinite-dimensional Ornstein-Uhlen\-beck process on the Fourier modes $U^{N}: = \{k \in \ZZ_M^d: \fF(\rho^{N,M})(k) = 0\}$, which is a strong Markov process in the space $C(\RR_+,\BB_{2,2}^{-\frac{d}{2}-})$ with invariant measure given by the projection of $\mu^M$ onto the Fourier modes in $U^{N}$. 
        
        On the other hand, $v^{N}$ is spectrally supported on $V_k^N$ and can therefore be identified with a finite-dimensional SDE with smooth coefficients. Thus, we only have to deal with the non-linear part $B^N$ to rule out explosions.  It suffices to show weak coercivity (see~\cite[Chapter~3]{liu2015stochastic}), which follows from 
        \begin{equation}\label{equ.preserve_L2}
            \langle u,B^N(u)\rangle_{L^2(\MM)} = 0.
        \end{equation}
        This follows from the divergence-free property of $u$ since
        \begin{equation*}
            \begin{aligned}
                \langle u,B^N(u)\rangle &= \langle \rho^N*u,\operatorname{div}((\rho^N*u)\otimes(\rho^N*u))\rangle \\
                &= \sum_{i,j} \langle\rho^N*u_i,\partial_j(\rho^N*u_i\cdot\rho^N*u_j)\rangle\\
                &=-\sum_{i,j}\langle (\partial_j\rho^N*u_i)\cdot\rho^N*u_j,\rho^N*u_i\rangle\\
                &=-\sum_{i,j}\langle \partial_j(\rho^N*u_i\cdot\rho^N*u_j),\rho^N*u_i\rangle = 0.
            \end{aligned}
        \end{equation*}
        This proves the existence and pathwise uniqueness of strong solutions in $C(\RR_+,\BB_{2,2}^{-\frac{d}{2}-})$.
        
        To prove invariance of $\mu^M$, we can either directly use the divergence-free property of $B^N$ when expressed in Fourier modes, or proceed as in \cite{gubinelli2013regularization,cannizzaro2021stationary} to calculate the generator of equation $\eqref{equ.Stochastic_NS_torus_truncated}$ and use the general result in \cite{echeverria1982criterion}. We postpone the argument until we introduced the generator of the equation $\eqref{equ.Stochastic_NS_torus_truncated}$. 
    \end{proof}
    
    To proceed, we introduce the generator of equation \eqref{equ.Stochastic_NS_torus_truncated} acting on cylinder functions. A cylinder function is a function with representation $F(u) = f(u(\varphi_1),\dots,u(\varphi_n))$ for some $n\in\NN$, $\varphi_1,\dots,\varphi_n \in \fS_f$ and smooth $f$ with at most polynomial growth at infinity. By It\^o's formula, the generator $\fL^N$ of equation $\eqref{equ.Stochastic_NS_torus_truncated}$ is given by $\fL^N  = \fL_\theta + \fG^N$.

        \begin{equation}\label{equ.geneator_Ltheta_cylindrical}
            \fL_\theta F(u) = -\sum_{i}\partial_i f \cdot u(A^\theta\varphi_i) + \sum_{i,j}\partial_{ij}^2f \langle A^{\frac{\theta}{2}}\varphi_i,A^{\frac{\theta}{2}}\varphi_j\rangle,
        \end{equation}
        and 
        \begin{equation}\label{equ.geneator_G_cylindrical}
            \tilde{\fG}^N F(u) = -\sum_i \partial_i f\langle B^N(u),\varphi_i\rangle,\qquad \fG^N = \lambda_{N,\theta}\tilde{\fG}^N.
        \end{equation}
    Using the approach of~\cite{gubinelli2020infinitesimal}, we calculate the action of the generator on the space $L^2(\mu^M)$.
    \begin{theorem}\label{thm.Geneator_in_Fourier}
        On the space $L^2(\mu^M)$, the operator $\fL_\theta$ is symmetric and the operator $\tilde \fG^N$ is anti-symmetric. Furthermore, we have the decomposition $\tilde{\fG}^N = \tilde{\fG}^N_+ + \tilde{\fG}^N_-$ such that $\tilde{\fG}^N_+:\FK_n\rightarrow\FK_{n+1}$ and  $\tilde{\fG}^N_-:\FK_n\rightarrow\FK_{n-1}$. 
        
        For sufficiently regular $\varphi = (\varphi_n)_{n\in\NN} \in \FK$, the operators have the following expressions in  Fourier modes for $n \ge 0$:
        \begin{equation}\label{equ.geneator_Ltheta_Fourier}
            \fF(\fL_\theta\varphi)_n((l_i,k_i)_{1:n}) = -(2\pi)^{2\theta}\sum_{i=1}^n|k_i|^{2\theta}\hat{\varphi}_n((l_i,k_i)_{1:n}),
        \end{equation}
        and $\tilde{\fG}_+^N(\FK_0) = 0$, $\tilde{\fG}_-^N(\FK_0) = 0$, while for $n \geq 1$
        \begin{equation}
            \begin{aligned}
         \fF(\tilde{\fG}_+^{N,M}\varphi)_n(l_{1:n},k_{1:n}) & = 2\pi\iota(n-1)\fR_{k_{n-1},k_n}^N(k_n+k_{n-1})_{l_n}\\
         &\qquad \cdot \hat{\varphi}_{n-1}((l_{n-1},k_n+k_{n-1}),\dots),
            \end{aligned}
        \end{equation}
        and
        \begin{equation}
            \begin{aligned}
                \fF(\tilde{\fG}_-^{N,M}\varphi)_n(l_{1:n},k_{1:n}) & =\frac{2\pi\iota n(n+1)}{M^d}\sum_{i}\sum_{p+q=k_n}\fR_{p,q}^N\Big[p_{l_n}\hat{\varphi}_{n+1}((i,p),(i,q),\dots) \\
                &\hspace{120pt} -k_{n,i}\hat{\varphi}_{n+1}((l_n,p),(i,q),\dots) \Big],
            \end{aligned}
        \end{equation}
        where $\fR_{p,q}^N = \hat{\rho}^N(p)\hat{\rho}^N(q)\hat{\rho}^N(p+q)$. 
    \end{theorem}
    
    \begin{proof}
        We divide the proof into three steps.
        The first two steps are used for the calculation of $\fL_\theta,\tilde{\fG}_+^N$ and $\tilde{\fG}_-^N$. The final step is to prove the anti-symmetry of $\tilde \fG^N$. Due to linearity, we only need to consider $F(u) = H_n(u(h)) = I_n(h^{\otimes h})$ with $h \in \fS_f(\MM)$ such that $\|h\|_{\HH_\MM}=1$.
        
        \noindent \textbf{Step 1.} We have for $\fL_\theta$:
        \begin{equation}
            \begin{aligned}
                \fL_\theta H_n(\mu^M(h)) &= -nH_{n-1}(\mu^M(h))\mu^M(A^\theta h) + n(n-1)H_{n-2}(\mu^M(h))\| A^{\frac{\theta}{2}}h\|^2_{\HH}\\
                &=-nI_n(h^{\otimes(n-1)}\otimes A^{\theta}h),
            \end{aligned}
        \end{equation}
        and thus for any $\varphi \in \FK$
        \[
            (\fL_\theta \varphi)_n = -\sum_{i=1}^n A_{x_i}^{\theta}\varphi_n,
        \]
        which gives the desired result in Fourier modes.
        
        \noindent \textbf{Step 2.} Defining $\rho_{i,y}^N(x)$ as a vector with value $\rho^N(x-y)$ at the $i-$th position and $0$ at other positions, we have 
        \begin{equation}
            \begin{aligned}
                    \tilde{\fG}^{N,M}F
                    &= -nH_{n-1}(u(h))\langle h,\rho^N*(\operatorname{div}(\rho^N*\mu^{M})\otimes(\rho^N*\mu^M))\rangle\\
                    &= n\sum_{i,j}\int \partial_jh_i(x)\rho^N(x-y)I_1(\rho_{i,y}^N)I_1(\rho_{j,y}^N)I_{n-1}(h^{\otimes(n-1)})dxdy \\
                    &= n\sum_{i,j}\int\partial_j h_i(x)\rho^N(x-y)I_2(\rho^N_{i,y}\otimes \rho_{j,y}^N) I_{n-1}(h^{\otimes(n-1)})dxdy\\
                    &= \tilde{\fG}_+^{N,M} + \tilde{\fG}_-^{N,M}.
            \end{aligned}
        \end{equation}
        The first equality holds because $h$ is divergence-free and the Leray projection is a Fourier multiplier, hence it commutes with the convolution. The third equality is because $I_1(\rho_{i,y}^N)I_1(\rho_{j,y}^N)=I_2(\rho^N_{i,y}\otimes \rho_{j,y}^N) + \langle\rho_{i,y}^N,\rho_{j,y}^N\rangle$ and
        \begin{equation*}
            \begin{aligned}
                &\sum_{i,j}\int \partial_jh_i(x)\rho^N(x-y)\langle\rho_{i,y}^N,\rho_{j,y}^N\rangle dxdy =\sum_i\int\partial_ih_i(x)\rho^N(x-y)\|\rho_y^N\|_{L^2}^2dxdy = 0.
            \end{aligned}
        \end{equation*}
        Using the contraction rule for Wiener-Itô integrals, see Proposition~1.1.3 of~\cite{Nualart2006}, $\tilde{\fG}_+^{N,M}$ and $\tilde{\fG}_-^{N,M}$ are given by (we drop the operator $I$ here) 
        \begin{equation*}
            \begin{aligned}
                \tilde{\fG}_+^{N,M}F(l_{1:n+1},x_{1:n+1}) & = \frac{n}{2}\int (\partial_{l_{n+1}}h_{l_n}(x) + \partial_{l_n}h_{l_{n+1}}(x))h^{\otimes(n-1)}(l_{1:n-1},x_{1:n-1})\\
            & \hspace{40pt} \cdot \rho^N_y(x)\rho^N_{z_n}(y)\rho^N_{z_{n+1}}(y)dxdy,
            \end{aligned}
        \end{equation*} 
        and 
        \begin{equation*}
            \begin{aligned}
                \tilde{\fG}_-^{N,M}F(l_{1:n-1},x_{1:n-1}) &= n(n-1)\int\sum_i(\partial_ih_{l_{n-1}}(x) +\partial_{l_{n-1}}h_i(x))h^{\otimes(n-2)}(l_{1:n-2},x_{1:n-2})\\
                &\hspace{90pt}\cdot h_i(z)\rho^N_y(x)\rho^N_z(y)\rho_{z_{n-1}}^N(y)dxdydz\\
                & = n(n-1)\int\sum_i(\partial_{l_{n-1}}h_i(x)\rho^{N}_{z_{n-1}}(y) - h_{l_{n-1}}(x)\partial_i\rho^N_{z_{n-1}}(y))\\
                &\hspace{70pt} \cdot h^{\otimes(n-2)}(l_{1:n-2},x_{1:n-2})h_i(z)\rho^N_y(x)\rho^N_z(y)dxdydz.
            \end{aligned}
        \end{equation*}
        The second equality holds because $h$ is divergence-free and it is useful for the estimations below. There is also a vanishing term:
        \begin{equation*}
            \begin{aligned}
                \tilde{\fG}_{-3}^{N,M}F(l_{1:n-3},x_{1:n-3}) &= \sum_{i,j}\int\partial_{j}h_{i}(x)\rho^N_y(x) h^{\otimes(n-3)}(l_{1:n-3},x_{1:n-3}) \langle h_{i},\rho^N_y\rangle\langle h_{j},\rho^N_y\rangle dxdy \\
                & = - \sum_{i,j}\int h_{i}(x)\rho^N_y(x)h^{\otimes(n-3)}(l_{1:n-3},x_{1:n-3})  \\
                &\hspace{50pt} \cdot \partial_{y_j} (\langle h_{i},\rho^N_y\rangle\langle h_{j},\rho^N_y\rangle) dxdy \\
                & = -\tilde{\fG}_{-3}^{N,M}F(l_{1:n-3},x_{1:n-3}),
            \end{aligned}
        \end{equation*}
        using that $h$ is divergence-free in the last step.


        A direct calculation shows that  $\tilde{\fG}_-^{N,M} = 0$ for $n=0,1$, and $\tilde{\fG}_+^{N,M} = 0$ for $n=0$. Hence, we have for all $\varphi \in \FK$
        \begin{equation}
            \begin{aligned}
                \tilde{\fG}_+^{N,M}\varphi_n & = \frac{n}{2}\int [\partial_{l_{n+1}}\varphi_n((l_n,x),\dots) + \partial_{l_n}\varphi_n((l_{n+1},x),\dots)]\\
            &\hspace{50pt}\cdot \rho^N_y(x)\rho^N_{z_n}(y)\rho^N_{z_{n+1}}(y)dxdy,
            \end{aligned}
        \end{equation}
        and
        \begin{equation}
            \begin{aligned}
                \tilde{\fG}_-^{N,M}\varphi_n =&n(n-1)\int \sum_i[\partial_{x_{l_{n-1}}}\varphi_n
                ((i,x)(i,z),\dots)\rho^{N}_{z_{n-1}}(y)\\
        &\hspace{40pt} - \varphi_n((l_{n-1},x),(i,z),\dots)\partial_i\rho^N_{z_{n-1}}(y)]\rho^N_y(x)\rho^N_z(y)dxdydz.
            \end{aligned}
        \end{equation}
        This gives the desired result on the Fourier modes of $\tilde{\fG}^{N,M}$.
        
    \noindent \textbf{Step 3.} Finally, we prove the anti-symmetry of $\tilde{\fG}^{N,M}$: We have for $h,g \in \fS_f$
    \begin{equation}
    \begin{aligned}
        &\frac{1}{(n+1)!}\langle\tilde{\fG}_+^{N,M}(h^{\otimes n}),g^{\otimes(n+1)}\rangle \\
        &\hspace{10pt} = n \langle g,h\rangle^{n-1} \int\sum_{l_{n:n+1}} g_{l_n}(z_n) g_{l_{n+1}}(z_{n+1})\rho^N_y(x)\rho^N_{z_n}(y)\rho^N_{z_{n+1}}(y) \\ 
        & \hspace{80pt}\cdot \partial_{l_{n+1}}h_{l_{n}}(x)dxdydz_{n:n+1} \\
        &\hspace{10pt} = -n \langle g,h\rangle^{n-1} \int\sum_{l_{n:n+1}}g_{l_n}(z_n) g_{l_{n+1}}(z_{n+1})\rho^N_y(x)\partial_{y_{l_{n+1}}}(\rho^N_{z_n}(y)\rho^N_{z_{n+1}}(y)) \\
        &\hspace{90pt} \cdot h_{l_{n}}(x)dxdydz_{n:n+1}\\
        &\hspace{10pt} =-n \langle g,h\rangle^{n-1} \int\sum_{l_{n},i} \partial_{i}g_{l_n}(x)g_{i}(z) \rho^N_y(x)\rho^N_{z_n}(y)\rho^N_{z}(y) h_{l_{n}}(x) dxdydzdz_{n}+0,
    \end{aligned}
    \end{equation}
    where we applied integration by parts in the second and third equalities and the second term in the third equality vanishes because $g$ is divergence-free. For the remaining part in the last line, we write $i = l_{n+1}, z = z_{n+1}$ and exchange $z_n$ and $x$. 
    For the operator $\tilde{\fG}_-^{N,M}$, we have 
    \begin{equation}
        \begin{aligned}
            \frac{1}{(n+1)!}\langle\tilde{\fG}_-^{N,M}(g^{\otimes (n+1)}),h^{\otimes n}\rangle 
            & =n\langle g, h\rangle^{n-1}\int\sum_{l_{n},i}h_{l_n}(z_n)g_i(z)\rho^N_y(x)\rho^N_z(y)\rho^N_{z_n}(y)\\
            &\hspace{70pt}\cdot (\partial_{l_n}g_i +\partial_i g_{l_n})(x)dxdydzdz_{n}\\
            & = 0 + \left(-\frac{1}{(n+1)!}\langle\tilde{\fG}_+^{N,M}(h^{\otimes n}),g^{\otimes(n+1)}\rangle\right).
        \end{aligned}
    \end{equation}
    %
    The first term after the second equality vanishes because
    \begin{equation*}
        \begin{aligned}
            &\int\sum_{l_n,i}h_{l_n}(z_n)g_i(z)\rho^N_y(x)\rho^N_z(y)\rho^N_{z_n}(y)\partial_{l_n}g_i(x)dxdydzdz_n\\
            &\hspace{30pt} = \int \sum_{l_n,i}(h_{l_n}*\rho^N)(y)\partial_{l_n}(g_i*\rho^N)^2(y)dydz_n = 0,
        \end{aligned}
    \end{equation*}
    because $h$ is divergence-free. Thus the operator $\fG^{N,M}$ is anti-symmetric.
    \end{proof}
    
    Now we can complete the proof of the Theorem \ref{thm.solution_truncated}, namely show that $\mu^M$ is an invariant measure for $u^{N,M}$.

    \begin{proof}[Proof of the invariance of $\mu^M$ in Theorem \ref{thm.solution_truncated}]
        Let $\Pi_K$ be the projection on the space generated by $\{\tilde e_k^i\}_{0\neq k \in \ZZ^d_M: \frac{|k|}{N}\le K}$. Then $\Pi_K$ commutes with $\fL^N$ if $K$ is sufficiently large, and the invariance of $\mu\circ \Pi_K^{-1}$ for $\Pi_K u^{N}$ follows from Echeverria's criterion~\cite{echeverria1982criterion} because
        \[
            \EE_{\mu \circ \Pi_K^{-1}}[\fL^N \Pi_K \varphi] = \EE_{\mu}[\fL^N \Pi_K \varphi] = 0,
        \]
        for any cylinder function $F$, which follows from Theorem~\ref{thm.Geneator_in_Fourier}.
    \end{proof}

    Next, we derive estimates for $\tilde{\fG}^{N}$, along the lines of~\cite{gubinelli2020infinitesimal, gubinelli2020hyperviscous}. We consider a function $\omega: \NN\rightarrow(0,\infty)$, which we interpret as a weight. 
    \begin{theorem}\label{thm.geneator_estimation}
        The operators $\tilde{\fG}_+^{N}$ and $\tilde{\fG}_-^{N}$ satisfy for $\zeta \in [0,1]$ and $\lambda_\theta^\zeta = \frac{d+2}{4\theta} + \zeta\frac{1\vee \theta-1}{2\theta}$
        \begin{equation}\label{eq:neg}
             \|\omega(\fN)(\lambda - \fL_\theta)^{\beta - \lambda^\zeta_\theta}\tilde{\fG}_-^{N}\varphi\| \lesssim\|\fN^{1 + \frac{1-\zeta}{2}(1-\frac{1}{\theta\vee 1})}\omega(\fN-1)(\lambda-\fL_\theta)^\beta\varphi\|,
        \end{equation}
        if $\beta > \frac{d}{4\theta}$, and 
        \begin{equation}\label{eq:pos}
             \|\omega(\fN)(\lambda - \fL_\theta)^{\beta - \lambda^\zeta_\theta}\tilde{\fG}_+^{N}\varphi\| \lesssim\|\fN^{1 + \frac{1-\zeta}{2}(1-\frac{1}{\theta\vee 1})}\omega(\fN+1)(\lambda-\fL_\theta)^\beta\varphi\|,
        \end{equation}
        if $\beta < \lambda_\theta^\zeta -\frac{d}{4\theta}$; the implicit constants in \eqref{eq:neg} and~\eqref{eq:pos} depend on $M$ if $\theta>1$, but they are independent of $M$ for $\theta \le 1$.
        Next, we also have for $0\leq \beta < \frac{d}{2\theta}$, 
        \begin{equation}
            \|(\lambda - \fL_\theta)^{-\frac{\beta}{2}}\fG_+^{N}\varphi\|^2
            \lesssim \lambda_N^2N^{d-2\beta\theta}\|\fN^{\frac{3}{2}-\frac{1}{2(\theta\vee 1)}}(-\fL_\theta)^{\frac{1}{2\theta}}\varphi\|^2,
        \end{equation}
        and for $\fG^N_-$, we have the estimate
        \begin{equation}
            \|(\lambda - \fL_\theta)^{-\frac{1}{2\theta}}\fG^{N}_-\varphi\|^2 \lesssim \|\fN(-\fL_\theta)^{1-\frac{1}{2\theta}}\varphi\|^2.
        \end{equation}
        Finally, we have a non-uniform estimate of the operator $\tilde{\fG}_{\pm}^{N}$. 
        \begin{equation}\label{equ.non_uniform_estimation}
            \|\tilde{\fG}^N\varphi\|\lesssim_N \|\fN^{\frac{3}{2}}\varphi\|.
        \end{equation}
    \end{theorem}
    
    \begin{proof}
        For $\tilde{\fG}_-^{N}$, we have, by the symmetry of $\hat{\varphi}$,
        \begin{equation*}
            \begin{aligned}
                &\|\omega(\fN)(\lambda - \fL_\theta)^{\beta - \lambda_\theta^\zeta}\tilde{\fG}_-^{N}\varphi\|^2 \\
                &\hspace{10pt}\lesssim \sum_{n\geq 0}n!\omega(n)^2\sum_{l_{1:n},k_{1:n}}(1 + (2\pi)^{2\theta}(|k_1|^{2\theta} +\dots+|k_n|^{2\theta}))^{2\beta-2\lambda_\theta^\zeta}\frac{1}{M^{d(n+2)}} \\
                &\hspace{10pt}\qquad \cdot n^2(n+1)^2\left|\sum_i\sum_{p+q = k_n}p_{l_n}\hat{\varphi}_{n+1}((i,p),(i,q), \dots) - k_{n,i}\hat{\varphi}_{n+1}((l_n,p),(i,q), \dots)\right|^2\\
                & \hspace{10pt} \lesssim \sum_{n\geq 0} \frac{n!n^2(n+1)^2\omega(n)^2}{M^{d(n+2)}}\sum_{l_{1:n},k_{1:n},i}(1 + |k_1|^{2\theta} +\dots+|k_n|^{2\theta})^{2\beta-2\lambda_\theta^\zeta}\\
                &\hspace{10pt}\qquad \cdot|k_n|^2\left(\left|\sum_{p+q=k_n}\hat{\varphi}_{n+1}((l_n,p),(i,q),\dots)\right|^2 + \left|\sum_{p+q=k_n}\hat{\varphi}_{n+1}((i,p),(i,q),\dots)\right|^2\right).
            \end{aligned}
        \end{equation*}
        By multiplying the RHS by the dimension $d$, we notice that 
        \begin{equation*}
        \sum_{l_{1:n},i}\left|\sum_{p+q=k_n}\hat{\varphi}_{n+1}((i,p),(i,q),\dots)\right|^2 \lesssim  d \cdot\sum_{l_{1:n},i}\left|\sum_{p+q=k_n}\hat{\varphi}_{n+1}((l_n,p),(i,q),\dots)\right|^2.
        \end{equation*}
        We have for any $\alpha > \frac{d}{2\theta}$
        \begin{equation}
            \frac{1}{M^d} \sum_{p+q = k_n}(C + |p|^{2\theta} + |q|^{2\theta})^{-\alpha} \approx \int (C+|k_n|^{2\theta} + |p|^{2\theta})^{-\alpha}dp \lesssim (C+|k_n|^{2\theta})^{\frac{d}{2\theta}-\alpha},
        \end{equation}
        and therefore
        \begin{equation*}
            \begin{aligned}     
            &\frac{1}{M^d}\left|\sum_{p+q=k_n}\hat{\varphi}_{n+1}((l_n,p),(i,q), \dots)\right|^2 \leq (\lambda + |k_1|^{2\theta} + \dots + |k_n|^{2\theta})^{\frac{d}{2\theta}-\alpha}\\
            &\hspace{110pt} \cdot \sum_{p+q = k_n}(\lambda + |k_1|^{2\theta} + \dots + |k_{n-1}|^{2\theta} + |p|^{2\theta} + |q|^{2\theta})^\alpha|\hat{\varphi}_{n+1}|^2.
            \end{aligned}
        \end{equation*}
        We then obtain, choosing $\alpha = 2\beta > \frac{d}{2\theta}$, 
        \begin{equation*}
            \begin{aligned}
                &\|\omega(\fN)(\lambda - \fL_\theta)^{\beta - \lambda_\theta^\zeta}\tilde{\fG}_-^{N}\varphi\|^2 \\
                & \hspace{20pt}\lesssim \sum_{n\geq 0} \frac{n!n^2(n+1)^2\omega(n)^2}{M^{d(n+1)}}\sum_{l_{1:n+1},k_{1:n}}(\lambda + |k_1|^{2\theta} +\dots+|k_n|^{2\theta})^{-2\lambda_\theta^\zeta +\frac{d}{2\theta}}|k_n|^2\\
                &\hspace{20pt}\qquad \cdot \sum_{p+q = k_n}(\lambda + |k_1|^{2\theta} + \dots + |k_{n-1}|^{2\theta} + |p|^{2\theta} + |q|^{2\theta})^{2\beta}|\hat{\varphi}_{n+1}((l_n,p),(l_{n+1},q),\dots)|^2 \\
                & \hspace{20pt} \lesssim \sum_{n\geq 0} \frac{1}{M^{dn}}n!n^{3 - \zeta -\frac{1-\zeta}{\theta\vee 1}}\omega(n-1)^2\sum_{(l,k)_{1:n}}(\lambda + |k_1|^{2\theta} + \dots + |k_{n}|^{2\theta})^{2\beta}|\hat{\varphi}_{n}|^2 \\
                &\hspace{20pt} = \|\fN^{1 + \frac{1-\zeta}{2}(1-\frac{1}{\theta\vee 1})}\omega(\fN - 1)(\lambda - \fL_\theta)^\beta\varphi\|^2.
            \end{aligned}
        \end{equation*}
        In the second inequality, we use symmetry of $\hat \varphi_n$ and $\sum_{i=1}^n |k_i|^2 \leq \left(\sum_{i=1}^n |k_i|^{2\theta}\right)^{\frac{1}{\theta}}$ for $\theta \leq 1$ while for $\theta > 1$, we use $|k_n|\ge \frac1M$ and therefore $|k_n|^2 \le M^{2\theta-2}$ to obtain 
        $\sum_{i=1}^n |k_i|^2 \lesssim_M \left(\sum_{i=1}^n |k_i|^{2\theta}\right)^{\zeta}\left(\sum_{i=1}^n |k_i|^{2\theta}\right)^{\frac{1-\zeta}{\theta}}n^{(1-\zeta)(1-\frac{1}{\theta})}$.
        
        The estimate~\eqref{eq:pos} for $\tilde{\fG}_+^{N}$ now follows by a duality argument from~\eqref{eq:neg} because $\langle \mathcal G^N_+ \varphi, \psi\rangle= -\langle \mathcal  \varphi, G^N_- \psi\rangle$ for all cylinder functions $\varphi$, $\psi$. 
        
        For $\fG_+^N$ (i.e. including $\lambda_N$) and $\beta < \frac{d}{2\theta}$ we have
        \begin{equation*}
        \begin{aligned}
            \|(\lambda - \fL_\theta)^{-\frac{\beta}{2}}\fG_+^{N,M}\varphi\|^2 & \lesssim \sum_{n\geq 0}\sum_{(l,k)_{1:n}}|\fR_{k_n,k_{n-1}}^N|^2\frac{n!\lambda_N^2}{M^{nd}}(\lambda + |k_1|^{2\theta}+\dots+|k_n|^{2\theta})^{-\beta}\\
                &\hspace{30pt}\cdot(n-1)^2|k_n+k_{n-1}|^2|\hat{\varphi}_{n-1}((l_{n-1},k_n+k_{n-1}),\dots)|^2 \\
            & \lesssim \lambda_N^2N^{d-2\beta\theta}\|\fN^{\frac{3}{2}-\frac{1}{2(\theta\vee 1)}}(-\fL_\theta)^{\frac{1}{2\theta}}\varphi\|^2,
        \end{aligned}
    \end{equation*}
    where the estimation in the second inequality follows from 
    \begin{equation}
            \frac{1}{M^d} \sum_{p+q = k} |\fR_{p,q}^N|^2(C+|p|^{2\theta} + |q|^{2\theta})^{-\beta} \lesssim \int_{|p|\leq N} (C+|p|^{2\theta})^{-\beta}dp\lesssim N^{d-2\beta\theta}.
    \end{equation}
    Finally, for $\fG_-^N$, we have
    \begin{equation}
    \begin{aligned}
        &\|(\lambda - \fL_\theta)^{-\frac{1}{2\theta}}\fG^{N,M}_-\varphi\|^2 \\
        &\hspace{20pt} \lesssim  \sum_{n\geq 0} \frac{n!n^2(n+1)^2}{M^{(n+2)d}}\sum_{(l,k)_{1:n}}\frac{\lambda_N^2|\hat{\rho}^N(k_n)|^2}{(\lambda + |k_1|^{2\theta} + \dots + |k_n|^{2\theta})^{\frac{1}{\theta}}}\\
        &\hspace{25pt}\cdot\left|\sum_i\sum_{\substack{p+q=k_n\\|p|,| q| \leq N}}p_{l_n}\hat{\varphi}_{n+1}((i,p),(i,q),\dots) - k_{n,i}\hat{\varphi}_{n+1}((i_{n-1},p),(i,q),\dots)\right|^2 \\
        &\hspace{20pt}\lesssim\sum_{n\geq 0} \frac{n!n^2(n+1)^2}{M^{(n+2)d}}\sum_{(l,k)_{1:n},i}\frac{\lambda_N^2|\hat{\rho}^N(k_n)|^2|k_n|^2\sum_{\substack{p+q=k_n\\|p|,| q| \leq N}}|\hat{\varphi}_{n+1}(p,q)|^2}{(\lambda + |k_1|^{2\theta} + \dots + |k_n|^{2\theta})^{\frac{1}{\theta}}}\\
        &\hspace{20pt} \lesssim \sum_{n\geq 0}\frac{(n+1)!n^3}{M^{(n+1)d}}\sum_{(l,k)_{1:n}}\lambda_{N}^2|\hat{\rho}^N(k_n)|^2\frac{1}{M^d}\sum_{\substack{p+q=k_n\\|p|,| q| \leq N}}|\hat{\varphi}_{n+1}(p,q,\dots)|^2.
    \end{aligned}
    \end{equation}
    We note furthermore that
    \begin{equation*}
        \begin{aligned}
            &\frac{1}{M^d}\sum_{\substack{p+q=k_n\\|p|,| q| \leq N}}|\hat{\varphi}_{n+1}(p,q,\dots)|^2 \\
            &\lesssim \frac{1}{M^d}\sum_{\substack{p+q=k_n\\|p|,| q| \leq N}}(|p|^{2\theta}+|q|^{2\theta})^{\frac{1}{\theta}-2}\sum_{p+q = k_n}(|p|^{2\theta}+|q|^{2\theta})^{2-\frac{1}{\theta}}|\hat{\varphi}_{n+1}(p,q,\dots)|^2\\
            &\lesssim  N^{d+2-4\theta}\sum_{p+q = k_n}(|p|^{2\theta}+|q|^{2\theta})^{2-\frac{1}{\theta}}|\hat{\varphi}_{n+1}(p,q,\dots)|^2,
        \end{aligned}
    \end{equation*}
    which gives the result since since $\lambda_{N}^2 = N^{4\theta - d-2}$. For the non-uniform estimation, we calculate as follows 
    \begin{equation}
        \begin{aligned}
            &\|\tilde{\fG}_+^N\varphi\|^2 \\
            &\lesssim \sum_{n\geq 0}\sum_{(l,k)_{1:n}}|\fR^N_{k_n,k_{n-1}}|^2\frac{n!(n-1)^2}{M^{nd}}|k_n+k_{n-1}|^2|\hat{\varphi}_{n-1}((l_{n-1},k_n+k_{n-1}),\dots)|^2\\
            &\lesssim \sum_{n\geq 0}\sum_{(l,k)_{1:n-1}}\frac{n!(n-1)^2}{M^{(n-1)d}}N^{d+2}|\hat{\varphi}_{n-1}((l_{n-1},k_{n-1}),\dots)|^2 \\
            &\lesssim_N \|\fN^{\frac{3}{2}}\varphi\|^2.
        \end{aligned}
    \end{equation}
    And for the operator $\tilde{\fG}_-$, we have 
    \begin{equation}
        \begin{aligned}
            &\|\tilde{\fG}_-^N\varphi\|^2 \\
            &\lesssim \sum_{n\geq 0} \frac{n!n^2(n+1)^2}{M^{(n+2)d}}\sum_{(l,k)_{1:n}}|\hat{\rho}^N(k_n)|^2\\
            &\hspace{20pt}\cdot\left|\sum_i\sum_{\substack{p+q=k_n\\|p|,| q| \leq N}}p_{l_n}\hat{\varphi}_{n+1}((i,p),(i,q),\dots) - k_{n,i}\hat{\varphi}_{n+1}((i_{n-1},p),(i,q),\dots)\right|^2\\
            &\lesssim_N \sum_{n\geq0}\frac{n!n^2(n+1)^2}{M^{(n+1)d}}\sum_{(l,k)_{1:n+1}}|\hat{\varphi}_{n+1}|^2 \lesssim \|\fN^{\frac{3}{2}}\varphi\|^2.
        \end{aligned}
    \end{equation}
    \end{proof}

    \begin{remark}
        The dependence of the implicit constant in~\eqref{eq:neg} and~\eqref{eq:pos} for $\theta>1$ should also appear in equation~(11) of~\cite{gubinelli2020hyperviscous}. Replacing the estimate $\sum_{i=1} |k_i|^2 \lesssim_M \sum_{i=1}^n |k_i|^{2\theta}$ by Hölder's inequality, we can get a uniform-in-$M$ estimate with $\fN^{\frac32-\frac{1}{2\theta}}$ on the right hand side, which therefore extends to $M=\infty$. But to prove well-posedness we need $\mathcal N^\gamma$ with some $\gamma \le 1$, which requires $\theta \le 1$. Therefore, the proof of well-posedness of energy solutions to the 2d hyperviscous Navier-Stokes equation in the plane in~\cite{gubinelli2020hyperviscous} seems incorrect. The more precise estimations in~\cite[Lemma 2.4]{cannizzaro2023gaussian} would give a bound with $\fN^{\frac32-\frac{1}{\theta}}$, thus $\theta \le 2$ would work. For $\theta > 2$ we are currently lacking an argument to prove well-posedness in $\RR^d$.
    \end{remark}

    \section{Existence and uniqueness of energy solutions for \texorpdfstring{$\theta > \frac{d}{2}$}{}}\label{sec:energy-solutions}

    To introduce energy solutions, we first introduce the crucial energy estimate.
    \begin{definition}
        We say that a process $(u_t)_{t\geq 0}$ with trajectories in $C(\RR_+,\fS'_f(\MM))$ satisfies an \emph{energy estimate} if for all $p\geq 1$ there exists $c_{2p}>0$ (with $c_{2} = 1$) such that for all cylinder functions $F$
        \begin{equation}\label{equ:ito_trick}
        \begin{aligned}
            \bE\left[\sup_{0\leq t\leq T}\left|\int_0^tF(s,u_s)ds\right|^p\right] \lesssim T^{\frac{p}{2}} \sup_{0\leq s\leq T} \|c_{2p}^{\fN}(1 - \fL_\theta)^{-\frac{1}{2}}F(s)\|^p.
        \end{aligned}
        \end{equation}
    \end{definition}
    \begin{lem}\label{lem.ito_trick}
        Let $\theta > 0$ and let $\nu \ll \mu$ be a probability measure with $\eta = \frac{d\nu}{d\mu^M} \in L^2(\mu^M)$. Let $u^{N,M}$ be the solution to the equation \eqref{equ.Stochastic_NS_torus_truncated} with initial condition $\nu$, then $u^{N,M}$ satisfies an energy inequality with implicit constant that does not depend on $N$.
    \end{lem}
    \begin{proof} This follows by the It\^o trick using the same arguments as in Section~4.1 of~\cite{gubinelli2020infinitesimal}.
     \end{proof}
    \begin{remark}
        The energy estimate holds also for the case $M=\infty$ if we have the existence of cylinder function martingale solutions of the stochastic Navier-Stokes equation on the whole space.
   \end{remark}

    As a simple consequence of the energy estimate we can make sense of additive functionals with distributional coefficients:
    \begin{lem}
        For any horizon $T$, assume that a process $(u_t)_{t \geq 0}$ satisfies \eqref{equ:ito_trick} and let $I_t(F) = \int_0^t F(u_s)ds $ for all cylinder functions $F$ and $t \in [0,T]$. Then the map $I$ can be uniquely extended as a bounded linear functional from $\fH^{-\frac{1}{2}}_0$ to $L^1(\Omega,\fF,\PP,C[0,T])$.
    \end{lem} 
    \begin{proof}
        We simply apply~\eqref{equ:ito_trick} with $p=1$ and use that cylinder functions are dense in $\fH^{-\frac{1}{2}}_0$.
    \end{proof}
    
    If $\theta > \frac{d}{2}$ and $\zeta=1$, then $\lambda^\zeta_\theta = \frac{d}{4\theta}+\frac12 < 1$ in equations~\eqref{eq:neg} and~\eqref{eq:pos}. Taking $\beta \in (\frac{d}{4\theta}, \frac12)$ and  $\omega(n) = \frac{1}{1+n}$ and using that $\mathcal \fH^\alpha_\gamma \subset \fH^{\alpha'}_{\gamma'}$ for $\alpha\ge \alpha'$ and $\gamma\ge \gamma'$, we obtain
    \begin{equation}\label{equ:neg_exp}
        \tilde{\fG}^N_{-}: \fH^{\frac{1}{2}}_0\rightarrow \fH^{-\frac{1}{2}}_{-1},\qquad \tilde{\fG}^N_{+}: \fH^{\frac{1}{2}}_0\rightarrow \fH^{-\frac{1}{2}}_{-1}.
    \end{equation}
    with bounds on the operator norms that are uniform in $N$. 
    With these  two bounds, we obtain the existence of the integral 
    \[
        \int_0^t \tilde \fL F(u_s) ds:=\lim_{N\rightarrow\infty}\int_0^t (\fL_\theta + \tilde \fG^N) F(u_s)ds,
    \]
    for all $(u_s)$ satisfying the energy estimate \eqref{equ:ito_trick} and for all cylinder functions $F$.
    
    We now introduce two different martingale problems for our equation. 
    \begin{definition}\label{def:cylindrical}
        A process $(u_t)_{t \geq 0}$ with trajectories in $C(\RR_+, \fS'_f(\MM))$ is an \emph{energy solution} for $\tilde \fL$ with initial distribution $\nu$ if $u_0 \sim \nu$ and if the following conditions are satisfied:
        \begin{itemize}
            \item $(u_t)_{t \geq 0}$ is incompressible, i.e. for all $T>0$ there exists $C(T)>0$ with
            \[
                \sup_{0\leq t\leq T} \EE|F(u_t)| \leq C(T)||F||,
            \]
            for all cylinder functions $F$;
            \item $u$ satisfies an energy estimate;
            \item for any cylinder function $F$, the process 
            \[
                M_t^F := F(u_t) - F(u_0) - \int_0^t \tilde \fL F(u_s)ds, \quad t\geq 0,
            \]
            is a continuous martingale with quadratic variation 
            \[
                \langle M_t^F\rangle = \int_0^t \fE(F)(u_s) ds,\quad \text{ with }  \fE(F):= 2\int_{\MM}|A_x^\frac{\theta}{2}D_xF|dx,
            \]
            where $D_x$ is the Malliavin derivative.
        \end{itemize}
    \end{definition}

    For cylinder functions $F$, we saw that $\tilde \fL F \in \fH^{-1/2}_0$ is only a distribution. But we will discuss below that there is a domain $\fD(\tilde \fL)$ such that $\tilde \fL \varphi \in \fH^0_0$ for all $\varphi \in \fD(\tilde \fL)$.
    \begin{definition}\label{def:martingale}
        A process $(u_t)_{t \geq 0}$ with trajectories in $C(\RR_+, \fS'_f(\MM))$ solves the martingale problem for $\tilde \fL$ with initial distribution $\nu$ if $u_0 \sim \nu$, $\operatorname{law}(u_t) \ll \mu^M$ for all $t\geq 0$ and if for all $\varphi \in \fD(\tilde \fL)$ and $t\geq 0$ we have $\int_0^t|\tilde \fL\varphi(u_s)|ds < \infty$ almost surely and the process
        \[
            \varphi(u_t) - \varphi(u_0) - \int_0^t \tilde \fL\varphi(u_s)ds, t\geq 0,
        \]
        is a martingale in the filtration generated by $(u_t)_{t\geq 0}$.
    \end{definition}

    To construct the domain $\fD(\tilde \fL)$, we follow the approach developed in Lukas Gr\"afner's Ph.D. thesis, which is presented in the lecture notes~\cite{grafner2023energy}. We need to show that $\fG$ is a bounded operator from $\fH^{1/2}_1$ to $\fH^{-1/2}_0$ (which follows from \eqref{equ:neg_exp}) and that the commutator $[\fN,\fG]$ is a bounded operator from $\fH^{1/2}_0$ to $\fH^{-1/2}_{-1}$. The commutator estimate follows in a similar way as \eqref{equ:neg_exp}, using $\omega(x) = 1$ this time as well as the fact that $\tilde \fG_+$ and $\tilde \fG_-$ only increase/decrease the chaos by one. Thus, we can apply Theorem~4.11 and Theorem~4.15 in \cite{grafner2023energy} to obtain uniqueness for the two martingale problems and the existence of the domain $(\fD(\tilde{\fL}),\tilde \fL)$ on which $\tilde \fL$ generates a contraction semigroup. Now we are ready to finish the proof of Theorem \ref{thm:exi_uni}.

    \begin{proof}[Proof of Theorem \ref{thm:exi_uni}]
        From Theorem 4.15 in \cite{grafner2023energy}, it remains to show the existence of energy solutions to \eqref{equ.Stochastic_NS_torus_truncated}.  The arguments are similar to Theorem~4.6 in \cite{gubinelli2020infinitesimal} or Theorem~1 in \cite{gubinelli2020hyperviscous}. Thus, we give only a sketch of the proof here.
        \begin{itemize}
            \item[1.] We start with the solution $u^N$ of the martingale problem of $\fL^N$, with the same initial distribution as for $u$. Given a cylinder function $F$ we consider the martingale
            \[
                M_t^F := F(u_t^N) - F(u^N_0) - \int_0^t (\fL_\theta + \tilde \fG^N) F(u_s^N)ds, \qquad t\geq 0,
            \]
            with quadratic variation $\int_0^t\fE(F)(u_s^N)$. 
            We furthermore obtain the equality 
            \[
                ||\omega(\fN)(\fE(F))^{\frac{1}{2}}||^2 = 2||\omega(\fN - 1)(-\fL_\theta)^{\frac{1}{2}}F||^2.
            \]
            \item[2.] From the decomposition 
            \[
                F(u^N_t) - F(u^N_s) = \int_s^t(\fL_\theta + \tilde \fG^N)(u^N_r)dr + M^F_t - M^F_s,
            \]
            we can obtain the estimation 
            \[
                \EE[|F(u^N)_t - F(u^N_s)|^p] \leq (t-s)^{\frac{p}{2}}||c_{4p}^\fN(1-\fL_\theta)^{-\frac{1}{2}}F||^p.
            \]
            Thus, we can apply the Kolomogorov and Mitoma’s criterion \cite{mitoma1983tightness} to get the tightness of the sequence $(u^N)_N$ in $C(\RR_+,\fS'_f)$.
            \item[1.] The incompressibility and energy inequality for $u^N$ follow directly from the Cauchy-Schwarz inequality and Lemma~\ref{lem.ito_trick}, respectively. The implicit constants are independent of $N$, and therefore any weak limit satisfies the same bounds.
            \item[4.] Using the uniform estimates in \eqref{equ:neg_exp}, we obtain that the $L^1$ distance between $\int_0^t(\fL_\theta + \tilde \fG^N)ds$ and $\int_0^t\tilde \fL(u^{N,M}_s)ds$ is $o(1)$ with respect to $N$. Thus, and limit point of $(u^N)_N$ as $N\to\infty$ satisfies also the third point in the definition of an energy solution.
        \end{itemize}
    \end{proof}

    \section{Solution on the whole space for $\theta \le 1$ and $N<\infty$}
    We consider here the truncated equation on the whole space $\RR^d$, i.e. when $M=\infty$, for $\theta \le 1$. To begin with, we derive the limit of the stationary distributions $\mu^M$.
    \begin{prop}
        The periodic extension $\tilde \mu^M$ of the white noise $\mu^M$ to the whole space converges weakly to $\mu$ as $M\rightarrow\infty$ as Gaussian random variables on $\fS'_f(\RR^d)$.
    \end{prop}
    \begin{proof}
        It is sufficient to prove that for each $f \in \fS_f(\RR^d)$, the random variable $\tilde \mu^M(f)$ converges weakly to $\mu(f)$. Since both are centered Gaussian, it suffices to consider the variance:
        \[
            \EE[\tilde{\mu}^M(f)^2] =  \int_{\left[-\frac{M}{2},\frac{M}{2}\right]^d} |f_M|^2 dx - \frac{1}{M^d}\left(\int_{\RR^d} fdx\right)^2,
        \]
        where $f_M(x) = \sum_{k\in\ZZ^d}f(x+kM)$. As $M\rightarrow\infty$, the second term  vanishes and the first term converges to $\|f\|_{\RR^d}^2$ by the dominated convergence theorem. Thus we get the result.
    \end{proof}
    
    To show the existence of weak solution to \eqref{equ.Stochastic_NS_torus_truncated} when $M =\infty$, we need some estimations.
    \begin{lem}\label{lem.u.power_est}
        Let $\theta\le 1$ and consider a weak solution $u^{N,M}$ to the equation \eqref{equ.Stochastic_NS_torus_truncated}. The following bounds hold uniformly in $M >0$: For all $\varphi \in \fS_f(\MM)$
        \begin{equation}
            \bE\left[\sup_{0 \leq t \leq T}\left|\int_0^t u^{N,M}(A^\theta\varphi)ds\right|^p\right] \lesssim T^{\frac{p}{2}}\|\varphi\|^p_{H^\theta},
        \end{equation}
        and 
        \begin{equation}
            \bE\left[\sup_{0\leq t\leq T}\left|\int_0^t B^N(u_s^{N,M})(\varphi)ds\right|^p\right] \lesssim \lambda_N N^{\frac{d-2\theta}{2}}T^{\frac{p}{2}}\|\varphi\|^p_{H^\theta}.
        \end{equation}
    \end{lem}
    \begin{proof}
        From Lemma \ref{lem.ito_trick}, we have 
        \begin{equation*}
                \bE\left[\sup_{0 \leq t \leq T}\left|\int_0^t u^{N,M}(A^\theta\varphi)ds\right|^p\right] \lesssim T^{\frac{p}{2}} \|c_{2p}^{\fN}(\lambda - \fL_\theta)^{-\frac{1}{2}}(A^\theta\varphi)\|^p \lesssim T^{\frac{p}{2}}\|\varphi\|^p_{H^\theta}.
        \end{equation*}
        Notice that for the functional $B^N$, we have 
        \begin{equation}\label{equ.non-linear_geneator}
            B^N(u^{N,M})(\varphi) = \fG^{N}_+\varphi. 
        \end{equation}
        Then, we obtain from Theorem $\ref{thm.geneator_estimation}$
        \begin{equation*}
                \bE\left[\sup_{0\leq t\leq T}\left|\int_0^t B^N(u_s^{N,M})(\varphi)ds\right|^p\right]\lesssim  T^{\frac{p}{2}} \|c_{2p}^{\fN}(\lambda - \fL_\theta)^{-\frac{1}{2}}(\fG_+^{N}\varphi)\|^p \lesssim \lambda_N N^{\frac{d-2\theta}{2}}T^{\frac{p}{2}}\|\varphi\|^p_{H^\theta}.
        \end{equation*}
    \end{proof}

    We are now able to show the existence of solutions.
    \begin{definition}
        Let $\nu\ll\mu$. We call a process $u^N$ with trajectories in $C(\RR_+, \fS_f'(\RR^d))$  a \emph{weak solution} of equation~\eqref{equ.Stochastic_NS_torus_truncated} with initial condition $\nu$ if for all $\varphi \in \fS_f(\RR^d)$
        \[
            u^{N}_t(\varphi) = \nu(\varphi) - \int_0^t u^{N}_s(A^\theta\varphi)ds + \lambda_{N,\theta}\int_0^tB^{N}(u^{N}_s)(\varphi)ds + M_t^{\varphi},
        \]
        where $M^\varphi$ is a continuous martingale with quadratic covariation
        \[
            \EE[M_t^\varphi M_s^{\phi}] = 2(t\wedge s)\langle A^\theta\Lp\varphi,A^\theta\Lp\phi\rangle_{\HH},
        \]
        for all $\phi,\varphi\in\fS_f(\RR^d)$. A weak solution is called an \emph{energy solution} if it is also incompressible and it satisfies an energy estimate, as in Definition~\ref{def:cylindrical} but with respect to $\mu$ instead of $\mu^M$.
    \end{definition}
    \begin{theorem}
        Let $\theta \in (0,1]$ and let $\eta \in \FK$ with $\eta \geq 0$ and $\int\eta d\mu = 1$. There exists a unique in law energy solution to the equation \eqref{equ.Stochastic_NS_torus_truncated} for $M = \infty$ with initial condition $d\nu = \eta d\mu$. 
    \end{theorem}
    \begin{proof}
        The existence proof is basically the same as the proof given in \cite{cannizzaro2021stationary} for the vorticity formulation of the two-dimensional case and very similar to the proof of Theorem~\ref{thm:exi_uni}. From the uniform-in-$M$ estimates in Lemma \ref{lem.u.power_est} and the Burkholder-Davis-Gundy inequality for the martingale, we deduce the tightness of $u^{N,M}$ by Kolomorgov's criterion. Because $u^{N,M}$ are energy solutions, we obtain the energy solution property of $u^N$ directly by taking the limit.
        
        The operator $\fL^N = \fL + \fG^N$ on the whole space has the same expression in Fourier modes as on the torus (see Theorem \ref{thm.Geneator_in_Fourier}), if we now consider frequencies in $\RR^d$. Because $\theta\le 1$, the operator $\mathcal G^N$ satisfies the same bounds as in Theorem~\ref{thm.geneator_estimation}. 
        Uniqueness in law then follows again by Theorem~4.15 in~\cite{grafner2023energy}.
    \end{proof}

    \section{Large-scale behavior for \texorpdfstring{$\theta \leq 1$}{}}\label{sec:large-scale}
    \subsection{Limiting behaviour}
    Now we consider $M=1$ or $M=\infty$, representing the torus or the whole space, $\theta \le 1$, and $N\to \infty$. For $\theta<1$, the limit is trivial:
    \begin{theorem}
        For $\theta < 1$, let $u^N$ be an energy solution to equation \eqref{equ.Stochastic_NS_torus_truncated} with initial distribution $\nu \ll\mu$ and $\frac{d\nu}{d\mu}\in L^2(\mu)$. As $N\to \infty$, $u^N$ converges to the unique-in-law weak solution of
        \begin{equation}\label{equ.stochastic_heat}
            \partial_t u = -(-\Delta)^{\theta} u + (-\Delta)^{\frac{\theta}{2}}\Lp\xi,\qquad u_0 \sim \nu.
        \end{equation}
    \end{theorem}
    \begin{proof}
        For any test function $\varphi\in\fS_f(\RR^d)$, the non-linear term is bounded by 
        \[
            \bE\left[\sup_{0\leq t\leq T}\left|\int_0^t B^N(u_s^{N})(\varphi)ds\right|^p\right] \lesssim  N^{\theta - 1}T^{\frac{p}{2}}\|\varphi\|^p_{H^\theta}.
        \]
        Since $\theta<1$,  this term will vanish when taking the limit $N\rightarrow\infty$. For the linear term, notice that we have tightness of $u^N$ by similar arguments as above. Since the linear part of the equation does not depend on $N$, any limit solves~\eqref{equ.stochastic_heat}. Uniqueness in law follows by similar arguments as for energy solutions but with $\tilde \fG=0$, or by testing against the heat kernel generated by $-(-\Delta)^\theta$.
    \end{proof}
    We now turn to the case $\theta = 1$. In this case, we have the following upper bound for the non-linear part:
        \[
            \bE\left[\sup_{0\leq t\leq T}\left|\int_0^t B^N(u_s^{N})(\varphi)ds\right|^p\right] \lesssim  T^{\frac{p}{2}}\|\varphi\|^p_{H^1},
        \]
    from where we can prove tightness as before. Following~\cite{cannizzaro20212d}, the aim is then  to derive a lower bound for the nonlinear part to show that its limit is non-zero.
    \begin{prop}
        Let $\theta = 1$, $\varphi\in\fS_f$, and let $u^N$ be the stationary solution to the equation \eqref{equ.Stochastic_NS_torus_truncated}, there exists a constant $C>0$ such that 
        \[
            \frac{C^{-1}}{\lambda^2}\|\varphi\|^2_{H^1}\leq \int_0^\infty e^{-\lambda t}\bE\left[\left|\int_0^t B^N(u_s^{N})(\varphi)ds\right|^2\right]dt \lesssim  \frac{C}{\lambda^2}\|\varphi\|^2_{H^1}.
        \]
    \end{prop}
    \begin{proof}
        We already showed the upper bound. For the lower bound, we apply Lemma 5.1 in \cite{cannizzaro20212d} to rewrite 
        \begin{equation*}
            \int_0^\infty e^{-\lambda t}\bE\left[\left|\int_0^t B^N(u_s^{N})(\varphi)ds\right|^2\right]dt = \frac{2}{\lambda^2}\EE[B^N(\varphi)(\lambda - \fL^N)^{-1}B^N(\varphi)],
        \end{equation*}
        where $B^N(\varphi) = \fG^N_+\varphi$ lives in the second Wiener chaos. Lemma 5.2 in the same paper gives a variational representation of the right hand side:
        \begin{equation*}
            =\frac{2}{\lambda^2}\sup_{G}\left\{2\langle \fG^N_+\varphi,G\rangle - \langle(\lambda - \fL_1)G,G\rangle - \langle\fG^N G,(\lambda-\fL_1)^{-1}\fG^N G\rangle\right\}.
        \end{equation*}
        For the third term, we divide $\fG^N = \fG^N_- + \fG^N_+$ and use the fact $(\lambda-\fL_1)^{-1}\fG^N_+\varphi$ is in the second Wiener chaos to obtain 
        \[
            = \frac{2}{\lambda^2}\sup_{G}\left\{2\langle \fG^N_+\varphi,G\rangle - \|(\lambda - \fL_1)^{\frac12}G\|^2 - \|(\lambda - \fL_1)^{-\frac{1}{2}}\fG^N_+G\|^2 - \|(\lambda - \fL_1)^{-\frac{1}{2}}\fG^N_-G\|^2\right\}.
        \]
        From now on we restrict to $G$ in the second chaos. Then we have by Theorem \ref{thm.geneator_estimation}
        \begin{equation*}
            \begin{aligned}
                 \|(\lambda - \fL_1)^{-\frac{1}{2}}\fG^N_+G\|^2 + \|(\lambda - \fL_1)^{-\frac{1}{2}}\fG^N_-G\|^2 \le C \|(\lambda-\fL_1)^{\frac{1}{2}}G\|^2, 
            \end{aligned}
        \end{equation*}
        and thus
        \[
            \ge  \frac{2}{\lambda^2}\sup_{G}\left\{2\langle \fG^N_+\varphi,G\rangle - (1+C) \|(\lambda - \fL_1)^{\frac12}G\|^2 \right\}.
        \]
        For 
        \[
            G = \delta(\lambda - \fL_1)^{-1}\fG^N_+\varphi,
        \]
        for $\delta>0$, we thus obtain the lower bound
        \begin{equation}
              \int_0^\infty e^{-\lambda t}\bE\left[\left|\int_0^t B^N(u_s^{N})(\varphi)ds\right|^2\right]dt \geq \frac{2}{\lambda^2}(2\delta - \delta^2(1+C)) \| (\lambda-\fL_1)^{-\frac{1}{2}}\fG^N_+\varphi\|^2, 
        \end{equation}
        and for sufficiently small $\delta>0$ the prefactor on the right hand side becomes positive. It remains to give a lower bound for $\|(\lambda -\fL_1)^{-\frac{1}{2}}\fG^N_+\varphi\|^2$. We have (note that for $M=\infty$ we have to replace all sums by integrals)
        \begin{equation*}
            \begin{aligned}
                \|(\lambda -\fL_1)^{-\frac{1}{2}}\fG^N_+\varphi\|^2 \gtrsim& \frac{\lambda_N^2}{M^{2d}}\sum_{l_1,l_2,k_1,k_2}(\lambda + |k_1|^2+|k_2|^2)^{-1}|\fR_{k_{1},k_2}^N|^2\\
                &\cdot |(k_1+k_2)_{l_2}\hat{\varphi}(l_1,k_1+k_2) + (k_1+k_2)_{l_1}\hat{\varphi}(l_2,k_1+k_2)|^2\\
                \ge &\frac{\lambda_N^2}{M^{2d}}\sum_{k_1,k_2}\frac{|\fR_{k_1,k_2}^N|^2|k_1+k_2|^2|\hat{\varphi}(k_1+k_2)|^2}{\lambda +|k_1|^2 + |k_2|^2}.
            \end{aligned}
        \end{equation*}
        For each $N,k$, define $k_N = \frac{k}{N}$ and 
        \[
            \vartheta^N(k) = \sum_{l+m = k}|\fR_{l,m}^N| \frac{N^{2-d}}{\lambda + |l|^2 + |m|^2},\quad \Theta^N_r(k_N) = \int\frac{\1_{B(0,r)\cap B(k_N,r)}}{\frac{\lambda}{N^2} + |x|^2 + |k_N - x|^2}dx.
        \]
        Then the lower bound can be expressed as 
        \begin{equation*}
            \int_{\ZZ_M^d}\vartheta^N(k)|k|^2|\hat{\varphi}(k)|^2 dk.
        \end{equation*}
        It is easy to check that we have 
        \[
            \Theta^N_1(k_N) \leq \vartheta^N(k) \leq \Theta^N_{1+\frac{1}{N}}(k_N),
        \]
        which shows that 
        \[
            \lim_{N\rightarrow\infty} \vartheta^N(k) = \int_{B(0,1)}\frac{1}{2|x|^2}dx.
        \]
        Then the dominated convergence theorem gives the desired lower bound.
    \end{proof}
    This proves the non-triviality of the non-linear term for the case $\theta = 1$.
    \begin{theorem}
        Let $\theta = 1$ and let $u^N$ be the stationary solution to the equation \eqref{equ.Stochastic_NS_torus_truncated}, then $u^N$ converges weakly to $u$ satisfying 
        \begin{equation}
            u_t(\varphi) = u_0(\varphi) + \int_0^t u(\Delta\varphi) + \fB(\varphi)_t + M_t^\varphi,
        \end{equation}
        where $\fB(\varphi)_t$ is such that
        \[
            \frac{1}{\lambda^2}\|\varphi\|^2_{\fH^1}\lesssim \int_0^\infty e^{-\lambda t}\bE[\left| \fB_t(\varphi)\right|^2]dt \lesssim  \frac{1}{\lambda^2}\|\varphi\|^2_{\fH^1}. 
        \]
    \end{theorem}
    Although so far we restricted our attention to dimension $d \geq 3$, the approach also works in dimension 2 and it yields the same conclusion as in \cite{cannizzaro2021stationary}, where the vorticity formulation of the same equation is studied, and triviality is shown for $\theta<1$.

    Next, we discuss possible approaches to describe the limit $\fB(\varphi)$ in the case $\theta=1$. Using the same method in \cite{cannizzaro2021weak}, we are able to determine the weak coupling limit for the two-dimensional stochastic Navier-Stokes equation.
    \begin{theorem}
        Let $d = 2$, $\theta = 1$ and the support of the Fourier transformation of mollifier $\rho^1$ is the indicator function of the unit ball. Let $u^N$ be the stationary solution to the equation
        \begin{equation}\label{equ.Stochastic_NS_torus_truncated_2d}
        \partial_tu^{N} = \Delta u^{N} - \frac{\hat{\lambda}}{\log N}B^N(u^{N}) + \sqrt{2}(-\Delta)^{\frac{1}{2}}\Lp\xi,
    \end{equation}
    on the torus $\TT^2$, then $u^N$ converges weakly to the stationary solution $u$ of
    \[
        \partial_t u = \nu_{\operatorname{eff}}\Delta u + \sqrt{2\nu_{\operatorname{eff}}}(-\Delta)^{\frac{1}{2}}\Lp\xi,
    \]
    where $\nu_{\operatorname{eff}} = \sqrt{1 + \frac{\hat{\lambda}^2}{2\pi}}$.
    \end{theorem}
    This can be deduced directly form the method introduced in the paper \cite{cannizzaro2021weak}. We only need to insert our formula for the operators $\fG^N_{\pm}$ into the replacement lemma therein to obtain the effective coefficient. 
    
    To generalize this result to higher dimensions, the method in \cite{cannizzaro2021weak} can no longer be applied due to the lack of control in the both off-diagonal part and diagonal part used for the proof of replacement lemma. We expect that the methods of~\cite{cannizzaro2023gaussian} can be applied to prove the following result:
    \begin{conjecture}\label{conj.limit_behaviour}
        Let $d\geq 3$, $\theta = 1$, $\omega_d = \frac{d\pi^{\frac{d}{2}}}{\Gamma(1+\frac{d}{2})}$ and assume that the support of the Fourier transformation of the mollifier $\rho^1$ is the indicator function of the unit ball. Let $u^N$ be the stationary solution to the  equation \eqref{equ.Stochastic_NS_torus_truncated} on $\TT^d$ or $\RR^d$, then it converges weakly to the stationary solution $u$ of
    \[
        \partial_t u = \nu_{\operatorname{eff}}\Delta u + \sqrt{2\nu_{\operatorname{eff}}}(-\Delta)^{\frac{1}{2}}\Lp\xi,
    \]
    where $\nu_{\operatorname{eff}} = \sqrt{1 +  \frac{\hat{\lambda}^2\omega_d}{4\pi^2(d-2)}}$.
    \end{conjecture}

    \subsection{Implication for fluctuating hydrodynamics}
    We revisit the fluctuating hydrodynamics equations introduced by Landau and Lifshitz in \cite{landau2013fluid}, which are given by
    \[
        \partial_t u = \nu\Delta u - \nabla p -\hat{\lambda} \operatorname{div}(u\otimes u) +  \nabla\cdot\tau, \qquad \nabla \cdot u = 0,
    \]
    where the centered Gaussian noise $\tau$ has covariance
    \begin{equation}
    \EE[\tau_{ij}(t,x)\tau_{kl}(t',x')] = \frac{2\nu k_B T}{\rho}(\delta_{ik}\delta_{jl}+\delta_{il}\delta_{jk} - \frac{2}{3}\delta_{ij}\delta_{kl})\delta(x-x')\delta(t-t').
\end{equation} 
Let us compute the covariance of $(\nabla\cdot\tau)_i$ for $i = 1,\dots,d$. We have 
\begin{equation*}
    \begin{aligned}
    &\EE\left[\sum_{j,k}\langle\partial_j\tau_{i_1,j},\varphi_1\rangle\langle\partial_k\tau_{i_2,k},\varphi_2\rangle\right] = \EE\left[\sum_{j,k}\langle\tau_{i_1,j},\partial_j\varphi_1\rangle\langle\tau_{i_2,k},\partial_k\varphi_2\rangle\right] \\
    =&\int \sum_{j,k} \frac{2\nu k_B T}{\rho}\left(\delta_{i_1i_2}\delta_{jk}+\delta_{i_1k}\delta_{ji_2} - \frac{2}{3}\delta_{i_1j}\delta_{i_2k}\right)\partial_j\varphi_1(t,x)\partial_k\varphi_2(t,x)dtdx.
    \end{aligned}
\end{equation*}
So if $i_1 \neq i_2$, we have   
\begin{equation*}
    \begin{aligned}
    &\EE\left[\sum_{j,k}\langle\partial_j\tau_{i_1,j},\varphi_1\rangle\langle\partial_k\tau_{i_2,k},\varphi_2\rangle\right] \\
    =& \int \frac{2\nu k_BT}{\rho}\left(\partial_{i_2}\varphi_1\partial_{i_1}\varphi_2 - \frac{2}{3}\partial_{i_1}\varphi\partial_{i_2}\varphi\right) dtdx \\
    =& \int \frac{2\nu k_BT}{3\rho}\partial_{i_1}\varphi_1\partial_{i_2}\varphi_2 dtdx \\
    =& \frac{2\nu k_BT}{3\rho}\EE[\langle\partial_{i_1}\tilde{\xi},\varphi_1\rangle\langle\partial_{i_2}\tilde{\xi},\varphi_2\rangle],
    \end{aligned}
\end{equation*}
where $\tilde{\xi}$ is a real-valued white noise on $\RR_+\times \RR^d$. When $i_1 = i_2$, we have 
\begin{equation*}
    \begin{aligned}
    &\EE\left[\sum_{j,k}\langle\partial_j\tau_{i_1,j},\varphi_1\rangle\langle\partial_k\tau_{i_1,k},\varphi_2\rangle\right] \\
    =& \int \frac{2\nu k_BT}{\rho}\left(\sum_j\partial_j\varphi_1\partial_j\varphi_1 + \frac{1}{3}\partial_{i_1}\varphi_1\partial_{i_1}\varphi_2\right) dtdx \\
    =& \frac{2\nu k_BT}{\rho}\EE[\langle(-\Delta)^{\frac{1}{2}}\xi_{i_1},\varphi_1\rangle\langle(-\Delta)^{\frac{1}{2}}\xi_{i_1},\varphi_2\rangle+\frac{1}{3}\langle\partial_{i_1}\tilde{\xi},\varphi_1\rangle\langle\partial_{i_1}\tilde{\xi},\varphi_2\rangle],
    \end{aligned}
\end{equation*}
where $\xi$ is a $d$-dimensional white noise on $\RR_+\times\RR^d$ that is independent of $\tilde \xi$.
Thus, we conclude that 
\begin{equation}
    \nabla\cdot\tau \stackrel{d}{=}\sqrt{\frac{2\nu k_BT}{\rho}}(-\Delta)^{\frac{1}{2}}\xi + \sqrt{\frac{2\nu k_BT}{3\rho}}\nabla\cdot\tilde{\xi}.
\end{equation}
Applying the Leray projection, we find that 
\[
    \Lp \nabla\cdot \tau \stackrel{d}{=} 
    \sqrt{\frac{2\nu k_BT}{\rho}}(-\Delta)^{\frac{1}{2}}\Lp\xi.
\]
Thus, the solution to the equation of fluctuating hydrodynamics should have the same distribution as the solution to the equation 
\[
    \partial_t u = \nu\Delta u - \nabla p -\hat{\lambda} \operatorname{div}(u\otimes u) +  \sqrt{\frac{2\nu k_BT}{\rho}}(-\Delta)^{\frac{1}{2}}\xi,\qquad \nabla \cdot u = 0,
\]
but of course we have to truncate both equations to make sense of them. Before that, we apply a scaling transformation to normalize the coefficients: Let $\tilde{u}(t,x) = Au(rt,x)$ with 
\[
    A = \sqrt{\frac{\rho}{k_B T}},\qquad r =\frac{1}{\nu}.
\]
The equation for $\tilde u$ becomes
\[
    \partial_t \tilde{u} = \Delta \tilde{u} - \nabla \tilde{p} -\hat{\lambda}\sqrt{\frac{k_BT}{\rho\nu^2}} \operatorname{div}(\tilde{u}\otimes \tilde{u}) +  \sqrt{2}(-\Delta)^{\frac{1}{2}}\xi,\qquad \nabla \cdot \tilde{u} = 0,
\]
and we consider the same truncation as before with $\rho^1$ decribed in the Conjecture~\ref{conj.limit_behaviour} (note that we did not rescale space yet and therefore the truncation appears in the same form in the equation for $u$):
\[
    \partial_t \tilde{u} = \Delta \tilde{u} - \nabla \tilde{p} -\hat{\lambda}\sqrt{\frac{k_BT}{\rho\nu^2}} \rho^1* \operatorname{div}(\rho^1*\tilde{u}\otimes \rho^1*\tilde{u}) +  \sqrt{2}(-\Delta)^{\frac{1}{2}}\xi,\qquad \nabla \cdot \tilde{u} = 0.
\]
Thus, Conjecture~\ref{conj.limit_behaviour} says that by considering the scaling 
\[
    \hat{u}^N(t,x) = N^{\frac{d}{2}}\tilde{u}(N^2t,Nx),
\]
the limit $u$ of $\frac{1}{A}\hat{u}^N(\frac{t}{r},x)$ will satisfy the equation 
\begin{align*}
    \partial_t u & = \frac{G(\hat{\lambda})}{r}\Delta u + \frac{\sqrt{2G(\hat{\lambda})}}{Ar^{\frac{1}{2}}}\Lp(-\Delta)^{\frac{1}{2}}\xi \\
    & = G(\hat \lambda) \nu \Delta u + \sqrt{G(\hat \lambda)}\sqrt{\frac{2\nu k_B T}{\rho}} \Lp(-\Delta)^{\frac{1}{2}}\xi,
\end{align*}
where 
\[
    G(\hat{\lambda}) = \sqrt{1 + \frac{\hat{\lambda}^2k_BT\omega_d}{4\nu^2\rho\pi^2(d-2)}}.
\]
    
\paragraph{Acknowledgement.} We are grateful to Ana Djurdjevac for helpful discussions on fluctuating hydrodynamics, and to Lukas Gr\"afner for helpful discussions on energy solutions. We acknowledge funding by DFG through IRTG 2544 ``Stochastic Analysis in Interaction''.
    
    \bibliographystyle{amsalpha}
	\bibliography{Triviality}

\providecommand{\bysame}{\leavevmode\hbox to3em{\hrulefill}\thinspace}
\providecommand{\MR}{\relax\ifhmode\unskip\space\fi MR }
\providecommand{\MRhref}[2]{%
  \href{http://www.ams.org/mathscinet-getitem?mr=#1}{#2}
}
\providecommand{\href}[2]{#2}
\begin{thebibliography}{BGME22}

\bibitem[BGME22]{bandak2022dissipation}
Dmytro Bandak, Nigel Goldenfeld, Alexei~A Mailybaev, and Gregory Eyink,
  \emph{Dissipation-range fluid turbulence and thermal noise}, Physical Review
  E \textbf{105} (2022), no.~6, 065113.

\bibitem[CES21]{cannizzaro20212d}
Giuseppe Cannizzaro, Dirk Erhard, and Philipp Sch{\"o}nbauer, \emph{2d
  anisotropic {KPZ} at stationarity: Scaling, tightness and nontriviality}, The
  Annals of Probability \textbf{49} (2021), no.~1, 122--156.

\bibitem[CET21]{cannizzaro2021weak}
Giuseppe Cannizzaro, Dirk Erhard, and Fabio Toninelli, \emph{Weak coupling
  limit of the anisotropic {KPZ} equation}, arXiv preprint arXiv:2108.09046
  (2021).

\bibitem[CGT23]{cannizzaro2023gaussian}
Giuseppe Cannizzaro, Massimiliano Gubinelli, and Fabio Toninelli,
  \emph{Gaussian fluctuations for the stochastic {B}urgers equation in
  dimension $ d\geq 2$}, arXiv preprint arXiv:2304.05730 (2023).

\bibitem[CK21]{cannizzaro2021stationary}
Giuseppe Cannizzaro and Jacek Kiedrowski, \emph{Stationary stochastic
  {N}avier-{S}tokes on the plane at and above criticality}, arXiv preprint
  arXiv:2110.03959 (2021).

\bibitem[Ech82]{echeverria1982criterion}
Pedro Echeverr{\'\i}a, \emph{A criterion for invariant measures of {M}arkov
  processes}, Zeitschrift f{\"u}r Wahrscheinlichkeitstheorie und Verwandte
  Gebiete \textbf{61} (1982), no.~1, 1--16.

\bibitem[FG95]{flandoli1995martingale}
Franco Flandoli and Dariusz Gatarek, \emph{Martingale and stationary solutions
  for stochastic {N}avier-{S}tokes equations}, Probability Theory and Related
  Fields \textbf{102} (1995), no.~3, 367--391.

\bibitem[FR08]{flandoli2008markov}
Franco Flandoli and Marco Romito, \emph{Markov selections for the 3d stochastic
  {N}avier--{S}tokes equations}, Probability Theory and Related Fields
  \textbf{140} (2008), no.~3, 407--458.

\bibitem[GHW23]{gess2023landau}
Benjamin Gess, Daniel Heydecker, and Zhengyan Wu,
  \emph{Landau-{L}ifshitz-{N}avier-{S}tokes equations: Large deviations and
  relationship to the energy equality}, arXiv preprint arXiv:2311.02223 (2023).

\bibitem[GIP15]{Gubinelli2015Paracontrolled}
Massimiliano Gubinelli, Peter Imkeller, and Nicolas Perkowski,
  \emph{Paracontrolled distributions and singular {PDE}s}, Forum of
  Mathematics, Pi \textbf{3} (2015), no.~e6.

\bibitem[GJ13]{gubinelli2013regularization}
Massimiliano Gubinelli and Milton Jara, \emph{Regularization by noise and
  stochastic {B}urgers equations}, Stochastic Partial Differential Equations:
  Analysis and Computations \textbf{1} (2013), no.~2, 325--350.

\bibitem[GJ14]{Goncalves2014}
Patr{\'\i}cia Gon{\c{c}}alves and Milton Jara, \emph{Nonlinear fluctuations of
  weakly asymmetric interacting particle systems}, Arch. Ration. Mech. Anal.
  \textbf{212} (2014), no.~2, 597--644.

\bibitem[GP18]{gubinelli2018energy}
Massimiliano Gubinelli and Nicolas Perkowski, \emph{Energy solutions of {KPZ}
  are unique}, Journal of the American Mathematical Society \textbf{31} (2018),
  no.~2, 427--471.

\bibitem[GP20]{gubinelli2020infinitesimal}
\bysame, \emph{The infinitesimal generator of the stochastic {B}urgers
  equation}, Probability Theory and Related Fields \textbf{178} (2020), no.~3,
  1067--1124.

\bibitem[GP23]{grafner2023energy}
Lukas Gr{\"a}fner and Nicolas Perkowski, \emph{Energy solutions and generators
  of singular {SPDE}s}, MPI MiS Leipzig Spring School 2023 (2023).

\bibitem[GT20]{gubinelli2020hyperviscous}
M~Gubinelli and M~Turra, \emph{Hyperviscous stochastic {N}avier--{S}tokes
  equations with white noise invariant measure}, Stochastics and Dynamics
  \textbf{20} (2020), no.~06, 2040005.

\bibitem[Hai14]{Hairer2014}
Martin Hairer, \emph{A theory of regularity structures}, Invent. Math.
  \textbf{198} (2014), no.~2, 269--504.

\bibitem[LL13]{landau2013fluid}
Lev~Davidovich Landau and Evgenii~Mikhailovich Lifshitz, \emph{Fluid mechanics:
  Landau and lifshitz: Course of theoretical physics, volume 6}, vol.~6,
  Elsevier, 2013.

\bibitem[LR15]{liu2015stochastic}
Wei Liu and Michael R{\"o}ckner, \emph{Stochastic partial differential
  equations: an introduction}, Springer, 2015.

\bibitem[Mit83]{mitoma1983tightness}
Itaru Mitoma, \emph{Tightness of probabilities on $c([0, 1]; y')$ and $d([0,
  1]; y')$}, The Annals of Probability (1983), 989--999.

\bibitem[Nua06]{Nualart2006}
David Nualart, \emph{The {M}alliavin calculus and related topics}, second ed.,
  Probability and its Applications (New York), Springer-Verlag, Berlin, 2006.
  \MR{2200233}

\bibitem[ZZ15]{zhu2015three}
Rongchan Zhu and Xiangchan Zhu, \emph{Three-dimensional {N}avier--{S}tokes
  equations driven by space--time white noise}, Journal of Differential
  Equations \textbf{259} (2015), no.~9, 4443--4508.

\end{thebibliography}
\end{document}